\documentclass[12pt]{amsart}
\usepackage[top =3.3cm, right=3.3cm, left=3.3cm, bottom=3.3cm]{geometry}
\usepackage{graphicx}
\usepackage{latexsym,amsfonts,amssymb,epsfig,verbatim}
\usepackage{amsmath, latexsym, graphics, textcomp}
\usepackage{amsthm}
\usepackage{pinlabel}
\usepackage{xcolor}
\usepackage{hyperref}
\usepackage{mathrsfs}

\newtheorem{Theorem}{Theorem}
\newtheorem{Corollary}{Corollary}
\newtheorem{Proposition}{Proposition}[section]
\newtheorem{Lemma}[Proposition]{Lemma}

\theoremstyle{definition}
\newtheorem{Definition}[Proposition]{Definition}
\newtheorem{Question}{Question}

\newtheorem*{Claim}{Claim}

\theoremstyle{remark}
\newtheorem{Remark}{Remark}

\newcommand{\len}{\operatorname{len}}

\title{Conformal tilings, combinatorial curvature, and the type problem}
\author{Mohith Raju Nagaraju}
\address{Department of Mathematics, Indian Institute of Science, Bangalore, India}
\email{mohith.math@gmail.com}


\begin{document}

\begin{abstract}
    Roughly, a conformal tiling of a Riemann surface is a tiling where each tile is a suitable conformal image of a Euclidean regular polygon. In 1997, Bowers and Stephenson constructed an edge-to-edge conformal tiling of the complex plane using conformally regular pentagons. In contrast, we show that for all $n\geq 7$, there is no edge-to-edge conformal tiling of the complex plane using conformally regular $n$-gons. More generally, we discuss a relationship between the combinatorial curvature at each vertex of the conformal tiling and the universal cover (sphere, plane, or disc) of the underlying Riemann surface. This result follows from the work of Stone (1976) and Oh (2005) through a rich interplay between Riemannian geometry and combinatorial geometry. We provide an exposition of these proofs and some new applications to conformal tilings.
\end{abstract}

\maketitle

\tableofcontents

\section{Introduction}

\subsection{Polygonal surfaces, combinatorial curvature, and complex structure}
A \textit{Polygonal surface} is obtained by considering a (possibly infinite) collection of Euclidean polygons and gluing their edges. Each edge is identified with exactly one other edge having the same edge length by an isometry of edges. The resulting object is a topological surface without boundary. Note that polygonal surfaces are sometimes called polyhedral surfaces. The dodecahedron consisting of pentagons and the hexagonal tessellation of the Euclidean plane with hexagons is an example of a compact and non-compact polygonal surface, respectively.

In a polygonal surface $E$, the interior of each polygon has a complex structure that is obtained by canonically embedding the polygon in the complex plane. When $E$ is an oriented surface, the complex structure in the interior of polygons extends to a global Riemann surface structure on the polygonal surface $E$ (see Subsection \ref{subsec:complex_structure}). A Riemann surface is called elliptic (resp. parabolic and hyperbolic) if its universal cover is the Riemann sphere (resp. complex plane and unit disc). Given a specific polygonal surface $E$, determining the universal cover of $E$ is called the \textit{type problem}.

A polygonal surface that contains only unit regular Euclidean polygons, that is, regular polygons with unit side length, is called a \textit{regular polygonal surface}. Such a surface is completely determined by its combinatorics. More precisely, the vertices and edges of the polygons give rise to an embedded graph and the combinatorics of this embedded graph completely determines the regular polygonal surface. In this article, we relate the local combinatorics of the embedded graph with the complex structure of the regular polygonal surface.

The local combinatorics around a vertex $v$ of a regular polygonal surface is captured by its \textit{vertex-type}. The vertex-type of a vertex $v$ is the cyclic tuple of integers $[k_1,k_2,\dots,k_d]$, where $d$ is the degree of $v$ and $k_i$ is the number of sides (the size) of the $i$-th polygon/tile around $v$, in either clockwise or counter-clockwise order (see Figure \ref{fig:vertex_type}). The information in the vertex-type of a vertex $v$ can be condensed into two important numbers: the \textit{angle-sum} $\mathcal{A}(v)$ at the vertex $v$ and the \textit{combinatorial curvature} $\kappa(v)$ at the vertex $v$. These two quantities are defined as:
\begin{equation}\label{eqn:angle_sum_combinatorial_curvature}
    \mathcal A(v):= \sum_{i=1}^d \left( \pi - \frac{2\pi}{k_i} \right) \qquad \quad \kappa(v):= \frac{2\pi - \mathcal A(v)}{2\pi}.
\end{equation}
Equivalently, the combinatorial curvature $\kappa(v)$ at $v$ can be expressed as:
\begin{equation}\label{eqn:combi_curv}
    \kappa(v):=  1 - \frac{\operatorname{deg}(v)}{2} + \sum_{i=1}^d \frac{1}{k_i},
\end{equation}
where $\operatorname{deg}(v)$ denotes the degree of the vertex $v$.
\begin{figure}[t!]
    \centering
    {
        \hfill
        \includegraphics[width=0.31\textwidth]{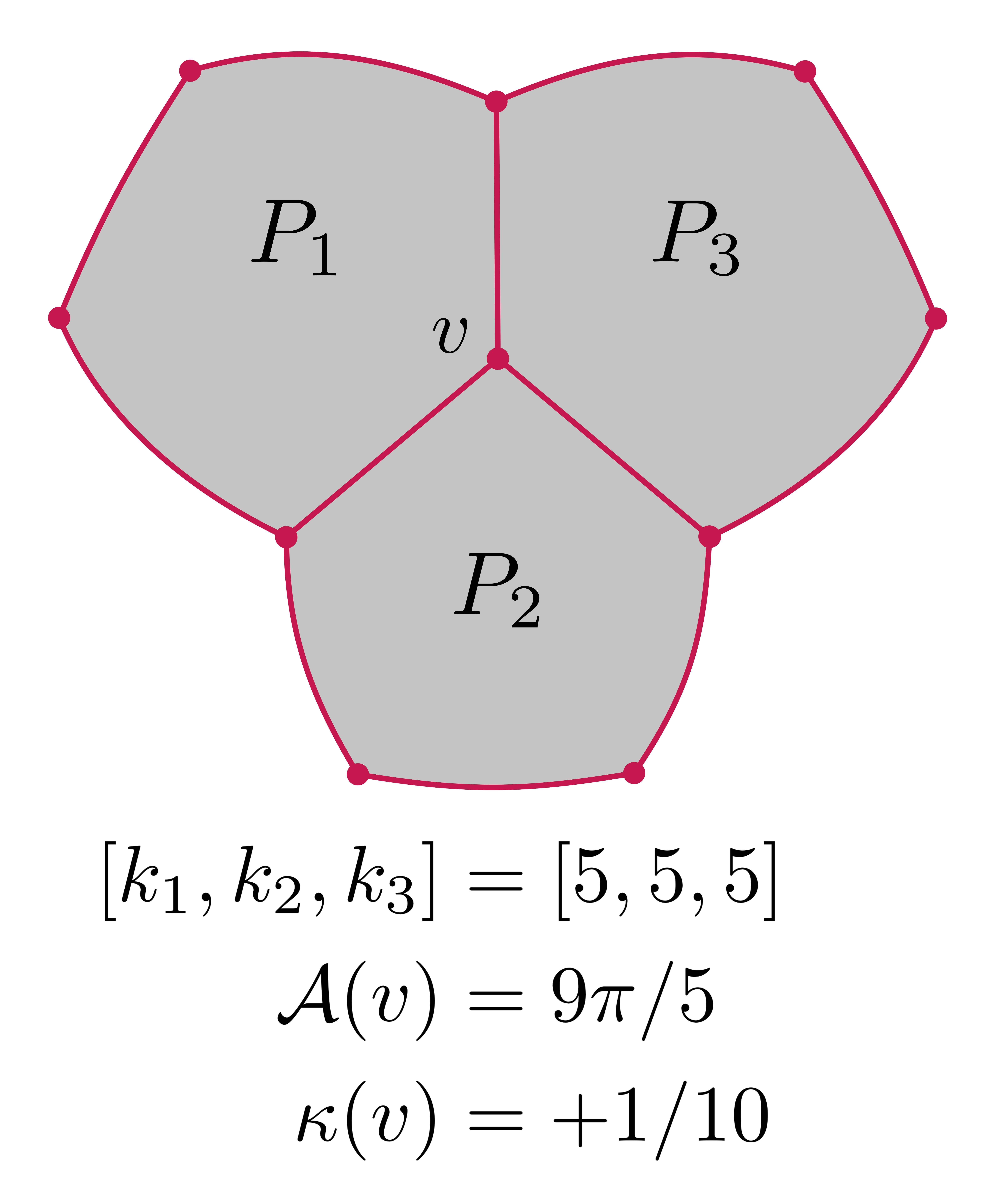} \hfill 
        \includegraphics[width=0.31\textwidth]{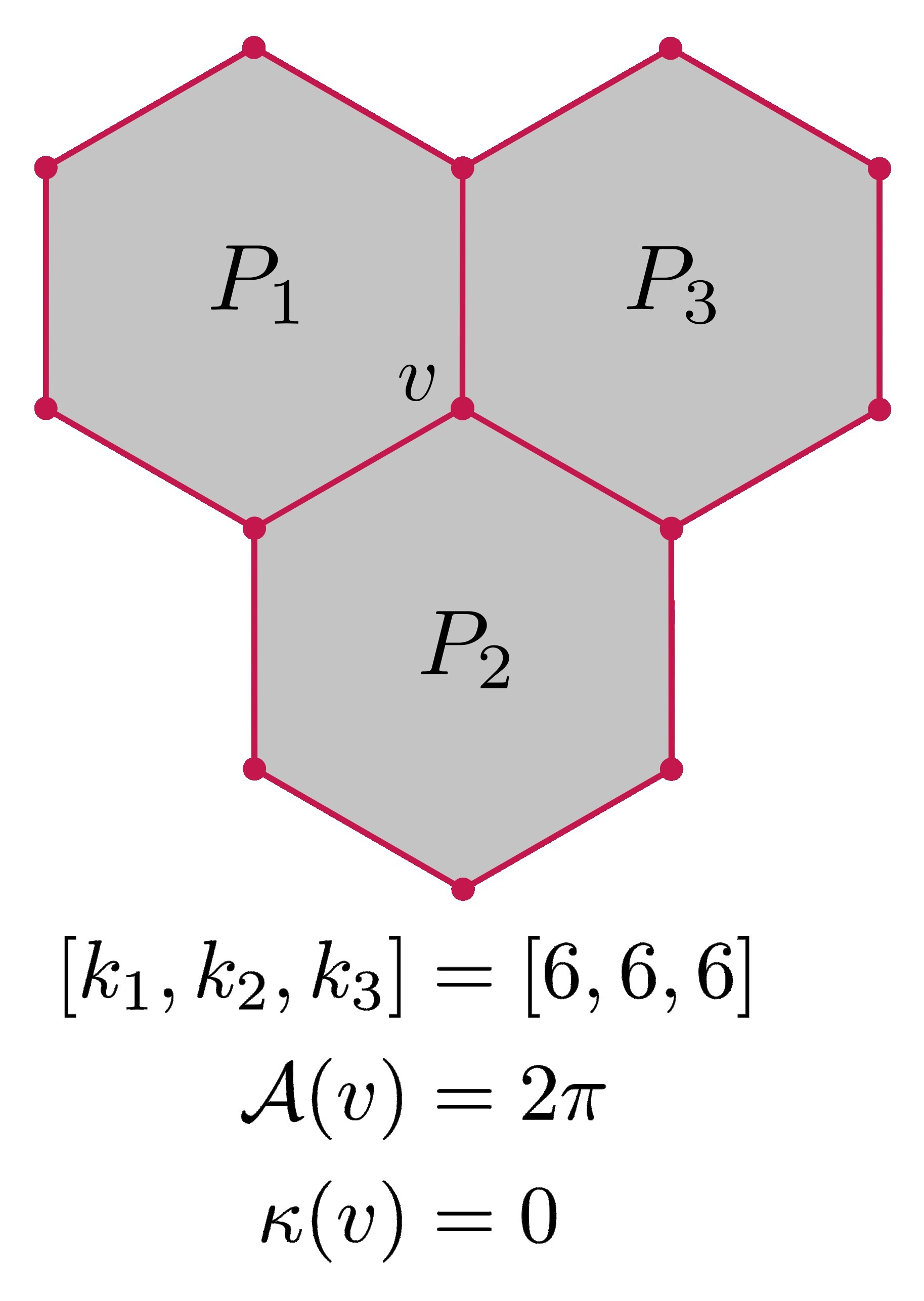}
        \hfill 
        \includegraphics[width=0.31\textwidth]{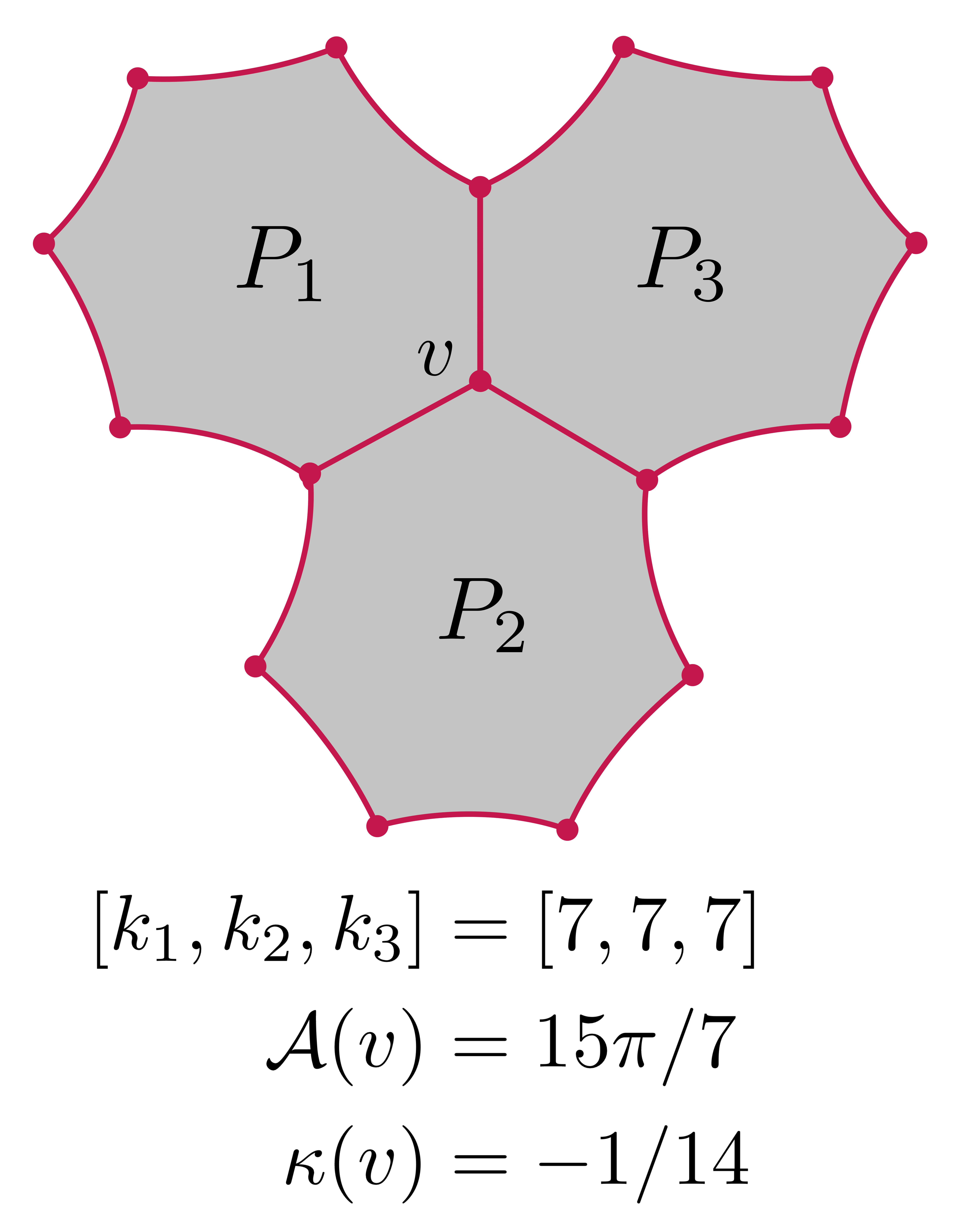} \hfill 
    }
    \caption{Three different vertex-types around a vertex $v$.}
    \label{fig:vertex_type}
\end{figure}

The following result relates the combinatorial curvatures of an oriented regular polygonal surface $E$ to the complex structure of $E$ under a mild extra assumption.
\begin{Theorem}\label{thm:Theorem_regular_polygonal_type}
    Suppose $E$ is an oriented regular polygonal surface. Let $\kappa(v)$ denote the combinatorial curvature at the vertex $v$ of $E$ and $\widetilde E$ be the universal cover of the Riemann surface $E$.
    \begin{enumerate}
        \item[(a)] If $\kappa(v)$ is strictly positive at each vertex and there is a constant $N$ such that each polygon in $E$ has at most $N$ sides, then $E$ is elliptic and $E = \hat{\mathbb C}$. 
        \item[(b)] If $\kappa(v)$ is zero at each vertex, then $E$ is parabolic and $\widetilde{E} = \mathbb C$.
        \item[(c)] If $\kappa(v)$ is strictly negative at each vertex and there is a constant $N$ such that each polygon in $E$ has at most $N$ sides, then $E$ is hyperbolic and $\widetilde{E} = \mathbb D$.
    \end{enumerate}
\end{Theorem}
To our knowledge, Theorem \hyperref[thm:Theorem_regular_polygonal_type]{1(a)} was first proved by Stone \cite{stone1976combinatorial,stone1976correction,stone1976geodesics}. Theorem \hyperref[thm:Theorem_regular_polygonal_type]{1(b)} boils down to a result in Riemannian geometry which states that if a simply-connected, complete $n$-dimensional Riemannian manifold is flat, then it is isometric to the Euclidean space $\mathbb R^n$. Theorem \hyperref[thm:Theorem_regular_polygonal_type]{1(c)} follows as a corollary of the work of Oh \cite{oh2005aleksandrov} on hyperbolicity of Aleksandrov surfaces. 

In Sections $3$ to $5$ we provide an exposition of the proof of Theorem \ref{thm:Theorem_regular_polygonal_type} closely following Stone and Oh. We point out that in \cite{stone1976geodesics}, Stone, in fact, proves a more general result for piece-wise linear manifolds of arbitrary dimension satisfying a suitable curvature positivity condition. An advantage of restricting to the special case of regular polygonal surfaces is that the proof simplifies greatly and illuminates the key ideas without getting distracted by the lengthy formalization of curvature for piece-wise linear manifolds.

We remark that the more recent work of DeVos and Mohar \cite{devos2007analogue} shows that the hypothesis of Theorem \hyperref[thm:Theorem_regular_polygonal_type]{1(a)} can be relaxed in two different directions. First, instead of requiring that the combinatorial curvature is strictly positive at each vertex, it suffices to assume that the combinatorial curvature is strictly positive at all but finitely many vertices. A different relaxation is that the assumption about the about the upper bound $N$ on the number of sides of a polygons can be removed. On the other hand, in the case of Theorem \hyperref[thm:Theorem_regular_polygonal_type]{1(c)}, we do not know whether the assumption about the upper bound $N$ on the number of sides of polygons can be removed.
\begin{Question}
    Is there an example of a regular polygonal surface where the combinatorial curvature at each vertex is strictly negative but whose universal cover is $\mathbb C$? (Note that such a surface must necessarily have polygons with arbitrarily large number of sides.)
\end{Question}
\begin{Remark}
    Theorem \hyperref[thm:Theorem_regular_polygonal_type]{1(a)} is a ``discrete" analog of the Bonnet-Myers theorem in Riemannian geometry which asserts that if a $2$-dimensional complete Riemannian manifold has Gaussian curvatures uniformly bounded below by a positive constant, then the manifold is compact; see, for example, \cite[Theorem 11.7]{lee2006riemannian}. On a similar note, Theorem \hyperref[thm:Theorem_regular_polygonal_type]{1(b)} and \hyperref[thm:Theorem_regular_polygonal_type]{1(c)} are analogous to a result of Milnor \cite{milnor1977deciding} which roughly states that if a $2$-dimensional simply-connected open Riemannian manifold is complete, rotationally symmetric, and has Gaussian curvatures $K$ that is sufficiently bounded below (resp. sufficiently negative), that is, $K\geq -1/(r^2 \log r)$ (resp. $K\leq -(1+\epsilon)/(r^2 \log r)$), then the manifold is conformally isomorphic to the complex plane $\mathbb C$ (resp. open unit disc $\mathbb D$) (see also \cite{doyle1988deciding}, \cite{grigor1985existence} and \cite{kenneth2002masters}). 
\end{Remark}
\subsection{Application to conformal tilings}
A \textit{conformal tiling} of a Riemann surface $X$ is a pair $\mathcal{T}=(E,g)$, where $E$ is an oriented regular polygonal surface and $g \colon E \rightarrow X$ is a conformal isomorphism between the two Riemann surfaces. The tiles of the conformal tiling $\mathcal{T}$ are the images $\{g(P_\alpha)\}_{\alpha \in I}$ of the regular Euclidean polygons $\{P_\alpha\}_{\alpha \in I}$ in the polygonal surface $E$. Conformal tilings were introduced by Bowers and Stephenson \cite{bowers1997regular,bowers2017conformal,bowers2019conformal}.%
\footnote{We point out that Bowers and Stephenson give two different complex structures for a conformal tiling: (i) reflective $\beta$-equilateral structure and (ii) regular piece-wise affine (RPWA) structure. We use the RPWA structure because in this case, the combinatorially defined angle-sum $\mathcal{A}(v)$ is concretely realized as the sum of angles of the polygons around the vertex $v$. Despite this choice, when there is an upper bound on the number of sides of the polygons/tiles in the conformal tiling, the reflective $\beta$-equilateral structure and the RPWA structure are quasi-conformal to each other (see \cite[Appendix A]{bowers2017conformal}). In particular, Theorem \ref{thm:Theorem_conformal_tiling_type} applies to both the complex structures under the extra assumption of the upper bound on the size of tiles.} %
As a first example, note that the regular polygonal surface $E$ obtained from the dodecahedron is conformally isomorphic to the Riemann sphere via a map $g \colon E \rightarrow \hat{\mathbb{C}}$ because of the uniformization theorem. Hence, the dodecahedron gives a conformal tiling $\mathcal{T}:= (E,g)$ of the Riemann sphere $\hat{\mathbb{C}}$. 

In \cite[Figure 1]{bowers1997regular}, the authors construct a striking example of a conformal tiling of the complex plane using ``conformally regular pentagons"; in our notation, the Riemann surface $X$ is $\mathbb C$ and the conformal tiling is $\mathcal{T}=(E,g\colon E \rightarrow \mathbb C)$, where $E$ is a regular polygonal surface consisting of only regular pentagons. This example shows that conformal tilings are somewhat more flexible than regular geometric tilings\footnote{A tiling of a surface $S$ is called a \textit{geometric tiling} if $S$ admits a Riemannian metric of constant curvature with respect to which each edge is a geodesic segment. Further, such a tiling is called \textit{regular} if each edge has the same length and all the interior angles of a given tile are equal (cf. \cite{edmonds1982regular}).} -- the Euclidean plane $\mathbb R^2$ does not have a tiling using regular pentagons, but there is a conformal tiling of the complex plane $\mathbb C$ using conformally regular pentagons. For more examples of conformal tilings, see \cite[Section 3]{bowers2017conformal} and \cite[Figure 1]{bishop2021non}.

Similar to regular polygonal surfaces, a conformal tiling $\mathcal{T}=(E,g)$ is completely determined, up to isomorphism, by the combinatorics of the embedded graph in $X$ obtained by taking the image of the edges and vertices of the regular polygonal surface $E$ under the homeomorphism $g \colon E \rightarrow X$. This graph can also be obtained by taking the boundaries of all the tiles in $X$, and is sometimes called the \textit{$1$-skeleton} of the tiling $\mathcal{T}$.

Given a conformal tiling $\mathcal{T}$ of a Riemann surface $X$, Theorem \ref{thm:Theorem_conformal_tiling_type} below establishes a relationship between the local combinatorics of the $1$-skeleton of $\mathcal{T}$ and the complex structure of the underlying Riemann surface $X$. Note that the vertices of the $1$-skeleton of a conformal tiling $\mathcal{T}=(E,g)$ is just the image of the vertices of $E$ under the conformal map $g$. Similarly, the combinatorial curvature at a vertex $v$ of $\mathcal{T}$ is just the combinatorial curvature at the vertex $g^{-1}(v)$ of $E$.
\begin{Theorem}\label{thm:Theorem_conformal_tiling_type}
    Suppose $\mathcal{T} = (E,g)$ is a conformal tiling of a Riemann surface $X$. Let $\kappa(v)$ denote the combinatorial curvature at the vertex $v$ of $\mathcal{T}$ and $\widetilde X$ be the universal cover of the Riemann surface $X$.
    \begin{enumerate}
        \item[(a)] If $\kappa(v)$ is strictly positive at each vertex and there is a constant $N$ such that each tile in $\mathcal{T}$ has at most $N$ sides, then $X$ is elliptic, and $X = \hat{\mathbb C}$. 
        \item[(b)] If $\kappa(v)$ is zero at each vertex, then $X$ is parabolic, and $\widetilde{X} = \mathbb C$.
        \item[(c)] If $\kappa(v)$ is strictly negative at each vertex and there is a constant $N$ such that each tile in $\mathcal{T}$ has at most $N$ sides, then $X$ is hyperbolic, and $\widetilde{X} = \mathbb D$.
    \end{enumerate}
\end{Theorem}
Theorem \ref{thm:Theorem_conformal_tiling_type} follows as a consequence of Theorem \ref{thm:Theorem_regular_polygonal_type}. To see this, let $\mathcal{T}=(E,g)$ be a conformal tiling of a Riemann surface $X$. Then, apply Theorem \ref{thm:Theorem_regular_polygonal_type} on the regular polygonal surface $E$ of $\mathcal{T}$. The result now follows because $E$ and $X$ are conformally isomorphic via the map $g \colon E \rightarrow X$.

Consider the special case where the Riemann surface $X$ of a conformal tiling $\mathcal{T}$ is homeomorphic to the plane $\mathbb R^2$. Now, Theorem \ref{thm:Theorem_conformal_tiling_type} gives a criterion for determining whether $X$ is conformally isomorphic to $\mathbb C$ or $\mathbb D$ based on the combinatorics of $\mathcal{T}$. This answers a special case of the ``type problem’’ for conformal tilings raised by Bowers and Stephenson \cite[Section 5.2]{bowers2017conformal}.

As an interesting corollary of Theorem  \ref{thm:Theorem_conformal_tiling_type}, we obtain that the complex plane and the Riemann sphere do not admit conformal tilings of certain combinatorics. A conformal tiling is called \textit{edge-to-edge} if the intersection of any two tiles contains at most one edge, or equivalently, if any two polygons in the underlying polygonal surface intersect along at most one edge.\footnote{The edge-to-edge condition is commonly used when classifying tilings and polyhedra. For example, the regular polygonal surface obtained by considering exactly two hexagons and gluing the edges pairwise is homeomorphic to the $2$-sphere but is not considered a Platonic solid because it is not edge-to-edge.} Note that in an edge-to-edge tiling, the degree at each vertex is at least $3$.
\begin{Corollary}\label{thm:Corollary_absence_of_tilings}
    \hfill
    \begin{enumerate}
        \item[(a)] The Riemann sphere $\hat{\mathbb{C}}$ does not admit an edge-to-edge conformal tiling where each tile has $6$ or more sides.
        \item[(b)] The complex plane $\mathbb{C}$ does not admit an edge-to-edge conformal tiling where each tile has $7$ or more sides.
    \end{enumerate}
\end{Corollary}
\begin{Corollary}\label{thm:Corollary_absence_of_triangulations}
    Suppose $\mathcal{T}$ is a conformal tiling of a Riemann surface $X$ where all the tiles have exactly $3$ sides, that is, all the tiles are conformal triangles. Let $\widetilde{X}$ be the universal cover of $X$. 
    \begin{enumerate}
        \item[(a)] If the degree at each vertex of $\mathcal{T}$ is strictly lesser than $6$, then $X$ is elliptic, and $X = \hat{\mathbb{C}}$.
        \item[(b)] If the degree at each vertex of $\mathcal{T}$ is exactly equal to $6$, then $X$ is parabolic, and $\widetilde{X} = \mathbb{C}$.
        \item[(c)] If the degree at each vertex of $\mathcal{T}$ is strictly greater than $6$, then $X$ is hyperbolic, and $\widetilde{X} = \mathbb{D}$.
    \end{enumerate}
\end{Corollary}
Corollary \ref{thm:Corollary_absence_of_tilings} and \ref{thm:Corollary_absence_of_triangulations} are proved by calculating the angle-sum and combinatorial curvature at each vertex. More precisely, Corollary \hyperref[thm:Corollary_absence_of_tilings]{1(a)} follows from a discrete analog of Gauss-Bonnet (Lemma \ref{thm:Lemma_Gauss_Bonnet}) and the others follow from Theorem \ref{thm:Theorem_regular_polygonal_type}. For example, to prove Corollary \hyperref[thm:Corollary_absence_of_tilings]{1(b)}, let $\mathcal{T}$ be an edge-to-edge conformal tiling of a Riemann surface $X$ where each tile has at least $7$ sides. In an edge-to-edge tiling the degree at each vertex is at least $3$. It follows that the vertex-type at each vertex of $\mathcal{T}$ is of the form $[k_1,k_2,\dots,k_d]$, where $d\geq 3$ and $k_1,k_2,k_3 \geq 7$. Using equation \ref{eqn:angle_sum_combinatorial_curvature}, note that the angle-sum of such a vertex-type is strictly greater than $2\pi$, or equivalently, the combinatorial curvature is strictly negative. Hence, by Theorem \ref{thm:Theorem_conformal_tiling_type} it follows that $X$ is hyperbolic. This shows that $X \neq {\mathbb C}$ and completes the proof of Corollary \hyperref[thm:Corollary_absence_of_tilings]{1(b)}.  

Corollary \ref{thm:Corollary_absence_of_tilings} is similar in spirit to the following facts: (a) there is no Archimedean solid in which each polygon has $6$ or more sides and (b) there is no Archimedean tiling of the Euclidean plane in which each polygon has $7$ or more sides. Even though conformal tilings are somewhat flexible, Corollary \ref{thm:Corollary_absence_of_tilings} and \ref{thm:Corollary_absence_of_triangulations} show that the combinatorial structure of the tiling imposes a strong influence on the complex structure of the underlying Riemann surface. In this spirit, we ask the following question.
\begin{Question}\label{thm:Question_n_4_surfaces}
    Given a natural number $n\geq 3$, which Riemann surfaces admit an edge-to-edge\footnote{Without the edge-to-edge condition, the set of Riemann surfaces that can be conformally tiled using $n$-sided tiles is the same as the set of Riemann surfaces that can be conformally tiled using triangles ($n=3$). This is because a Riemann surface $X$ has a conformal tiling using only $n$-sided tiles if and only if $X$ admits a Belyi function and the latter does not depend on the number $n$ (see \cite[Proposition 2.7]{bishop2021non}). The edge-to-edge condition constricts the set of Riemann surfaces that admit a conformal tiling using only $n$-sided tiles (cf. Corollary \ref{thm:Corollary_absence_of_tilings}).} conformal tiling using only $n$-sided tiles?
\end{Question}
For triangles ($n=3$), this question has been answered completely. The answer comes in two pieces -- the compact case and the non-compact case. For the compact case, note there are only countably many compact regular polygonal surfaces and hence there is at most a countable number of compact Riemann surfaces that have a conformal tiling using triangles. Furthermore, Belyi's theorem provides a complete characterization: a compact Riemann surface $X$ admits a conformal tiling using triangles if and only if $X$ can be defined as the zero loci of a set of polynomials having coefficients in the algebraic numbers $\overline{\mathbb{Q}}$ (see \cite[Theorem 4]{belyi1979galois} and \cite[Theorem 1.3]{jones2016dessins}). For the non-compact case, the recent work of Bishop and Rempe \cite{bishop2021non} shows that every non-compact Riemann surface has a conformal tiling using triangles. Question \ref{thm:Question_n_4_surfaces} asks for analogous results when $n\geq 4$. We aim to address this question in a forthcoming article.
\begin{Remark}
    Conformal tilings are closely related to the notion of circle packings, whose study was initiated by Thurston (see \cite{bowers2017conformal}).  Corollary \ref{thm:Corollary_absence_of_triangulations} is analogous to the results of He and Schramm in the study of circle packing which assert parabolicity (resp. hyperbolicity) of circle packings if the degree at each vertex is lesser than or equal to $6$ (resp. strictly greater than $6$) \cite[Theorem 10.1 and 10.2]{he1995hyperbolic}; see also \cite[Theorem 1]{beardon1991circle}. Furthermore, in the same article, He and Schramm provide a general parabolicity and hyperbolicity criterion based on whether the contact graph of the circle packing satisfies certain isoperimetric inequalities. The proof of Theorem \hyperref[thm:Corollary_absence_of_tilings]{1(b)} in this article also proceeds by establishing an isoperimetric inequality.
\end{Remark}

\noindent\textbf{Organization of the article.} Section \ref{sec:preliminaries} discusses three preliminaries: the complex structure and conformal metric of an oriented regular polygonal surface and the existence of universal coverings of regular polygonal surfaces. The proof of Theorem \ref{thm:Theorem_regular_polygonal_type} proceeds case by case depending on whether the combinatorial curvature is positive, zero, or negative. Each case is discussed in a separate section (Sections \ref{sec:positive_case}, \ref{sec:zero_case}, and \ref{sec:negative_case}).

\medskip
\noindent\textbf{Acknowledgements.} This article arose from a question raised by my advisor, Subhojoy Gupta, about the generalization of the result of Bishop and Rempe (2021) to other regular polygons. I am grateful to Subhojoy Gupta for generously sharing his time and ideas and closely guiding me throughout this project. I thank Kishore Vaigyanik Protsahan Yojana (KVPY) and Innovation in Science Pursuit for Inspired Research (INSPIRE) for the undergraduate fellowship and contingency grant. Also, I thank Aritra Chatterjee, Sumanta Das, and Ajay Kumar Nair for helpful discussions.

\section{Preliminaries}\label{sec:preliminaries}

\noindent{\textbf{A technical remark and some notation.}} Recall that a polygonal surface $E$ is defined to be $\cup_\alpha P_\alpha/{\sim}$ where each $P_\alpha$ is a Euclidean polygon and $\sim$ denotes the identification of the edges. In particular, $P_\alpha$ is a subset of $\mathbb R^2$. The key point is that inclusion $\iota: P_\alpha \hookrightarrow E= \cup_\alpha P_\alpha/{\sim}$ might not be injective because of the identification of edges. For example, a polygonal surface homeomorphic to the torus can be constructed by identifying the opposite sides of a square. In this case, the inclusion of the square in the polygonal surface is not injective. However, the inclusion $\iota: \operatorname{Int}(P_\alpha) \hookrightarrow E= \cup_\alpha P_\alpha/{\sim}$ is always an injective embedding. Hence, we use $\operatorname{Int}(P_\alpha)$ interchangeably with its image $\iota(\operatorname{Int}(P_\alpha))$ in $E$. We caution that in many cases $\iota(\operatorname{Int}(P_\alpha)) \neq \operatorname{Int}(\iota(P_\alpha))$. Also, given a tile $T$ of a tiling, we write $\operatorname{Int}(T)$ to mean the set obtained by removing the vertices and edges of $T$; note that this might be smaller than the topological interior $T \setminus \partial T$. 

Similarly, given an edge $e$ of a Euclidean polygon $P_\alpha$, the inclusion $\iota \colon e \hookrightarrow E= \cup_\alpha P_\alpha/{\sim}$ may not be injective, but the inclusion $\iota \colon \operatorname{Int}(e) \hookrightarrow E= \cup_\alpha P_\alpha/{\sim}$ is always an injective embedding. Thus, we use $\operatorname{Int}(e)$ interchangeably with $\iota(\operatorname{Int}(e))$. Also, given an edge $e$ of a tiling or a graph embedded in a surface, we write $\operatorname{Int}(e)$ to mean the set obtained by removing the vertices $e$; note that this is different from the topological interior. We extend this notion further: given an edge $e$ in $E$ that is adjacent to $P_\alpha$ and $P_\beta$, we write $\operatorname{Int}(P_\alpha \cup e \cup P_\beta)$ to mean the union $\operatorname{Int}(P_\alpha) \cup \operatorname{Int}(e) \cup \operatorname{Int}(P_\beta)$. Again, we caution that $\operatorname{Int}(P_\alpha \cup e \cup P_\beta)$ might be different from the topological interior of $P_\alpha \cup e \cup P_\beta$. 

\subsection{Complex structure of an oriented regular polygonal surface}\label{subsec:complex_structure} 
The complex structure on an oriented regular polygonal surface $E$ is specified by the following collection of complex charts.
\begin{figure}[t!]
    \centering
    {
        \includegraphics[width=0.7\textwidth]{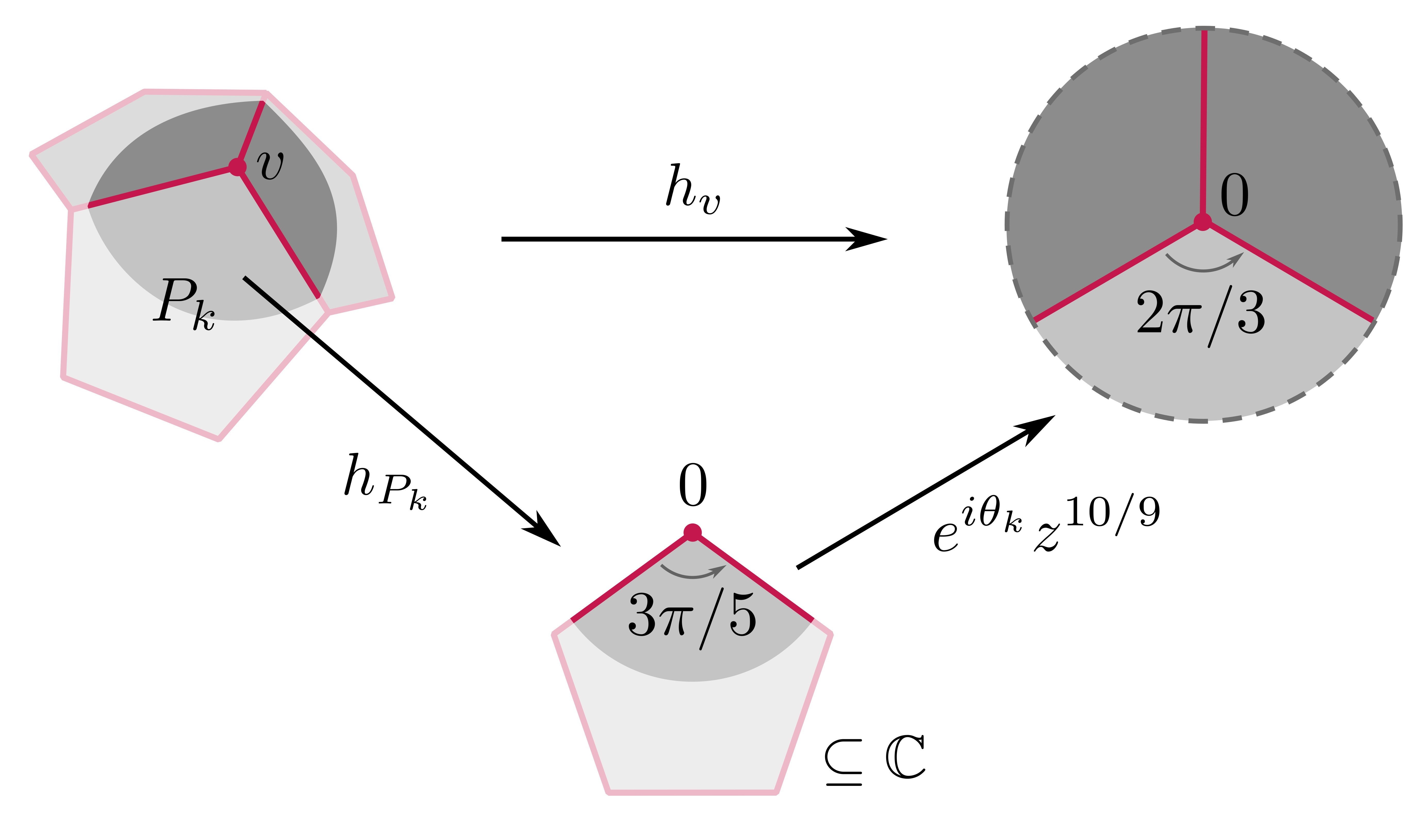}
    }
    \caption{Complex chart at a vertex with vertex-type $[5,5,5]$ defined using ``power maps".}
    \label{fig:vertex_chart}
\end{figure}
\begin{itemize}
    \item For each polygon $P_\alpha$ in $E$, the set $\operatorname{Int}(P_\alpha)$ can be canonically embedded in $\mathbb C$, because each unit regular polygon $P_\alpha$ is a closed subset of $\mathbb R^2$. This defines a chart $h_\alpha \colon \operatorname{Int}(P_\alpha) \rightarrow \mathbb C$. Of course, post-composing $h_\alpha$ with any orientation-preserving isometry of $\mathbb C = \mathbb R^2$ gives another compatible chart. We mention that $h_\alpha$ is chosen so that it is orientation-preserving with respect to the orientations of $E$ and $\mathbb C$.
    \item Let $e$ be an edge that is adjacent to $P_\alpha$ and $P_\beta$. We define a chart $h_{\alpha\beta} \colon \operatorname{Int}(P_\alpha \cup e \cup P_\beta) \rightarrow \mathbb C$ by combining the two charts $h_\alpha \colon \operatorname{Int}(P_\alpha) \rightarrow \mathbb C$ and $h_\beta \colon \operatorname{Int}(P_\beta) \rightarrow \mathbb C$. More elaborately, post compose $h_\alpha$ by an orientation-preserving isometry to obtain another chart $h'_{\alpha}$ that satisfies  $\lim_{x \rightarrow x_0} h'_\alpha(x)= \lim_{x \rightarrow x_0} h_\beta(x)$ for all points $x_0 \in \operatorname{Int}(e)$. This condition along with the fact that $h'_\alpha$ and $h_\beta$ are orientation-preserving allows us to glue $h'_\alpha$ and $h_\beta$ to obtain a chart $h_{\alpha\beta} \colon \operatorname{Int}(P_\alpha \cup e \cup P_\beta) \rightarrow \mathbb C$.
    \item Lastly, charts can be defined at the vertices using \textit{power maps} $z \mapsto z^{2\pi/\mathcal{A}(v)}$ that normalize the cone angle $\mathcal{A}(v)$ at $v$ to the angle $2\pi$ (see Figure \ref{fig:vertex_chart}). For a more elaborate description, let $v$ be a vertex, $\mathcal A(v)$ be the angle-sum at $v$, and $P_1,P_2,\dots,P_d$ be the polygons cyclically arranged around $v$. Next, choose a small neighborhood $N_v$ of $v$ in $E$ and define $h_{v,k} \colon N_v\cap \operatorname{Int}(P_k) \rightarrow \mathbb C$ as 
    \[h_{v,k}(z):= e^{i\theta_k}z^{2\pi/\mathcal{A}(v)}.\]
    Here, we have identified the domain $N_v\cap \operatorname{Int}(P_k)$ with a subset of $\mathbb C$ using $h_{P_k}$. Moreover, $h_{P_k}$ is chosen such that the vertex $v$ corresponds to $0$ in $\mathbb C$. Hence, $N_v\cap \operatorname{Int}(P_k)$ is the intersection of an open neighborhood of $0$ with an open sector. Next, the angle $\theta_k$ are chosen suitably so that the union $\bigcup_k h_{v,k}(N_v\cap \operatorname{Int}(P_k))$ fills up a neighborhood of $0$ and $h_{v,k} \colon N_v\cap \operatorname{Int}(P_k) \rightarrow \mathbb C$ glue together to give a chart $h_v \colon N_v \rightarrow \mathbb C$. 
\end{itemize}

\subsection{Conformal metric on certain regular polygonal surfaces}
\begin{Lemma}\label{thm:Lemma_flat_conformal_metric}
    Suppose $E$ is an oriented regular polygonal surface with angle-sum equal to $2\pi$ at each vertex. Then, there is a conformal metric $\rho$ on $E$ which restricts to the Euclidean metric $|dz|^2$ on the interior of each polygon $\operatorname{Int}(P_\alpha)$ of $E$. Moreover, $\rho$ is a flat metric.
\end{Lemma}
\begin{proof}[Proof of Lemma \ref{thm:Lemma_flat_conformal_metric}]
    The complex structure of $E$ is described in Subsection \ref{subsec:complex_structure} using charts $h_\alpha, h_{\alpha,\beta},h_v$ whose image is in the complex plane $\mathbb C$. Equip $\mathbb C$ with the standard Euclidean metric $|dz|^2$. Now, using the definitions of the charts $\{h_\alpha, h_{\alpha,\beta},h_v\}$, observe that the transition maps between these charts are isometries of $\mathbb C$. We point out that the hypothesis $\mathcal{A}(v)=2\pi$ is used here to show the transition map $h_v\circ h^{-1}_{\alpha}$ is an isometry; in fact, $h_v\circ h^{-1}_{\alpha}(z)=a+e^{i\theta_\alpha}z^{2\pi/\mathcal{A}(v)}$ is an isometry if and only if $\mathcal{A}(v)=2\pi$.

    Now, we can define a flat metric $\rho$ on $E$ by pulling back the metric $|dz|^2$ of $\mathbb C$ using the collection of charts $\{h_\alpha, h_{\alpha,\beta},h_v\}$. This is well-defined as the transition maps are isometries. In conclusion, $\rho$ is a smooth conformal metric that agrees with the Euclidean metric $|dz|^2$ on each $\operatorname{Int}(P_\alpha)$.
\end{proof}
In the proof of Theorem \hyperref[thm:Theorem_regular_polygonal_type]{1(c)}, we shall need a conformal metric on regular polygonal surfaces whose angle-sums are not zero. Towards this, we have the following Lemma.
\begin{Lemma}\label{thm:Lemma_conformal_metric}
    Suppose $E$ is an oriented regular polygonal surface with angle-sum at least $2\pi$ at each vertex. Then, there is a conformal metric $\rho$ on $E$ with certain singularities. More precisely, in any holomorphic chart, the metric tensor $\rho$ can be written as $\varrho(z)^2|dz|^2$, where $\varrho(z)$ is a continuous function that is smooth and positive away from the vertices of $E$. Further, the metric $\rho$ restricts to the Euclidean metric $|dz|^2$ on the interior of each polygon $\operatorname{Int}(P_\alpha)$ of $E$
\end{Lemma}
\begin{proof}[Proof of Lemma \ref{thm:Lemma_conformal_metric}]
    Proceed as in Lemma \ref{thm:Lemma_flat_conformal_metric} and consider the pullback of $|dz|^2$ from $\mathbb{C}$ to $E \setminus V$ using the collection of charts $\{h_\alpha, h_{\alpha,\beta}\}$; here, $V$ stands for the set of vertices in $E$ and the vertex charts $\{h_v\}$ are momentarily excluded. This defines a smooth conformal metric $\rho$ on $E\setminus V$ that agrees with the Euclidean metric $|dz|^2$ on each $\operatorname{Int}(P_\alpha)$. However, the pullback metric $h^*_v(|dz|^2)$, using the vertex chart $h_v \colon N_v \rightarrow \mathbb C$, does not agree with the metric $\rho$ on the overlapping regions $N_v\setminus \{v\}$. 
    
    To remedy this, equip $\mathbb C$ with the metric $$\varrho(z)^2|dz|^2:=|z|^{(\mathcal{A}(v)-2\pi)/\pi}|dz|^2.$$
    Observe that $\varrho(z)$ is continuous as $\mathcal{A}(v)\geq 2\pi$ at every vertex; further, it is smooth and positive away from $0$. Now, consider the pullback metric $h^*_v(\varrho(z)^2|dz|^2)$ on the neighborhood $N_v$ of $v$. It turns out that $h^*_v(\varrho(z)^2|dz|^2)$ agrees with $\rho$ on overlapping regions. To see this, recall the transition function $h_v \circ {h_{\alpha\beta}^{-1}} \colon h_{\alpha\beta}(\operatorname{Int}(P_\alpha) \cap N_v) \rightarrow h_v(\operatorname{Int}(P_\alpha) \cap N_v)$ is given by $z\mapsto e^{i\theta_k}z^{2\pi/\mathcal{A}(v)}$ and note
    \begin{align*}
        (h_v \circ {h_{\alpha\beta}^{-1}})^*(\varrho(z)^2|dz|^2) &= \varrho(h_v \circ {h_{\alpha\beta}^{-1}}(z))^2\cdot |(h_v \circ {h_{\alpha\beta}^{-1}})'(z)|^2\cdot |dz|^2\\
        &= |z^{2\pi/\mathcal{A}(v)}|^{(\mathcal{A}(v)-2\pi)/\pi}\cdot |z^{(2\pi-\mathcal{A}(v))/\mathcal{A}(v)}|^2\cdot |dz|^2\\
        &= |dz|^2
    \end{align*}
    In conclusion, the metric $\rho$ along with $h^*_v(\varrho(z)^2|dz|^2)$ define a continuous conformal metric on $E$ that is smooth and positive away from the vertices.
\end{proof}

\subsection{Universal cover of a polygonal surface}
\begin{Lemma}\label{thm:Lemma_lift_polygonal_surface}
    Suppose $E$ is a polygonal surface. Then, there exists a polygonal surface $\widetilde E$ which is a universal cover of $E$ along with a universal covering map $\pi \colon \widetilde E \rightarrow E$. Furthermore, the map $\pi$ satisfies the following properties.
    \begin{enumerate}
        \item[(a)] The vertices, edges, and polygons of $\widetilde E$ maps under $\pi$ to vertices, edges, and polygons of $E$, respectively.
        \item[(b)] The image of an $n$-sided polygon of $\widetilde E$, under the map $\pi$, is an $n$-sided polygon of $E$.
        \item [(c)] The vertex-type of a vertex $\tilde v$ in $\widetilde E$ is equal to the vertex-type of the vertex $\pi(\tilde v)$ in $E$. In particular, the combinatorial curvature $\kappa(\tilde v)$ is equal to the combinatorial curvature $\kappa(\pi(\tilde v))$.
    \end{enumerate}
\end{Lemma}
\begin{proof}
    Given a polygonal surface $E$, let $\{P_\alpha\}_{\alpha \in I}$ be the collection of Euclidean polygons that make up $E$. As $E$ is a topological surface, there exists a universal cover $S$ along with a universal covering map $q \colon S \rightarrow E$. A \textit{tile} of $S$ is defined to be the closure of a connected component of $q^{-1}(\operatorname{Int}(P_\alpha))$ for some $\alpha \in I$. The collection of all tiles in $S$ is denoted by the collection $\{T_\beta\}_{\beta \in J}$. Next, given an index $\beta \in J$, suppose $P_{\alpha(\beta)}$ is the polygon such that $\operatorname{Int}(P_{\alpha(\beta)}) = q(\operatorname{Int}(T_\beta))$. Then, define $Q_\beta$ to be an isometric copy of the polygon $P_{\alpha(\beta)}$. This gives a collection of polygons $\{Q_\beta\}_{\beta \in J}$.

    Observe that there exists a map $f \colon \bigcup_{\beta \in J} Q_\beta \rightarrow S$ such that (i) $f(Q_\beta) = T_\beta$ and (ii) $q \circ f \colon \operatorname{Int}(Q_\beta) \rightarrow q(\operatorname{Int}(T_\beta))=\operatorname{Int}(P_{\alpha(\beta)})$ is an isometry between the interiors of the two Euclidean polygons $Q_\beta$ and $P_{\alpha(\beta)}$. To see this, first note that for each polygon $P_\alpha$, we have $\operatorname{Int}(P_\alpha) \subseteq E$ is simply connected and hence the covering map $q$ maps each component of $q^{-1}(\operatorname{Int}(P_\alpha))$ homeomorphically onto $\operatorname{Int}(P_\alpha)$, that is, $q \colon \operatorname{Int}(T_\beta) \rightarrow q(\operatorname{Int}(T_\beta))$ is a homeomorphism for each $\beta$. Now, given a polygon $Q_\beta$, recall that $\operatorname{Int}(Q_\beta)$ is isometric to $q(\operatorname{Int}(T_\beta)) = \operatorname{Int}(P_{\alpha(\beta)})$ via some isometry $g_\beta \colon \operatorname{Int}(Q_\beta) \rightarrow q(\operatorname{Int}(T_\beta))$. Next, define $f$ on the domain $\operatorname{Int}(Q_\beta)$ to be $f|_{\operatorname{Int}(Q_\beta)}:=q^{-1}\circ g_\beta$ and then extend $f$ by continuity to the closed polygon $Q_\beta$. This map $f$ satisfies conditions (i) and (ii) above.

    Lastly, we claim $( \bigcup_{\beta \in J} Q_\beta )/f$ is a polygonal surface $\widetilde E$ and $q\circ f \colon \widetilde{E} \rightarrow E$ is the required universal covering map. The proof of the claim has four parts: (i) $f$ is injective on the interiors $\{\operatorname{Int}(Q_\beta)\}_{\beta \in J}$, (ii) $f$ identifies an edge $e'_1$ of $Q_\beta$ with exactly one other edge $e'_2$ of some polygon $Q_\beta'$, (iii) the quotient space $( \bigcup_{\beta \in J} Q_\beta )/f$ is a polygonal surface $\widetilde E$ and is homeomorphic to $S$, and (iv) the universal covering map $\pi:= q \circ f \colon \widetilde{E} \rightarrow E$ has properties (a) to (d) stated in Lemma \ref{thm:Lemma_lift_polygonal_surface}. Part (i) follows because $f(\operatorname{Int}(Q_\beta)) = \operatorname{Int}(T_\beta)$ and $\{\operatorname{Int}(T_\beta)\}_{\beta \in J}$ are distinct components of $q^{-1}(\cup_\alpha(\operatorname{Int}(P_\alpha)))$.
    \begin{figure}[t!]
        \centering
        {
            \hfill
            \includegraphics[width=0.3\textwidth]{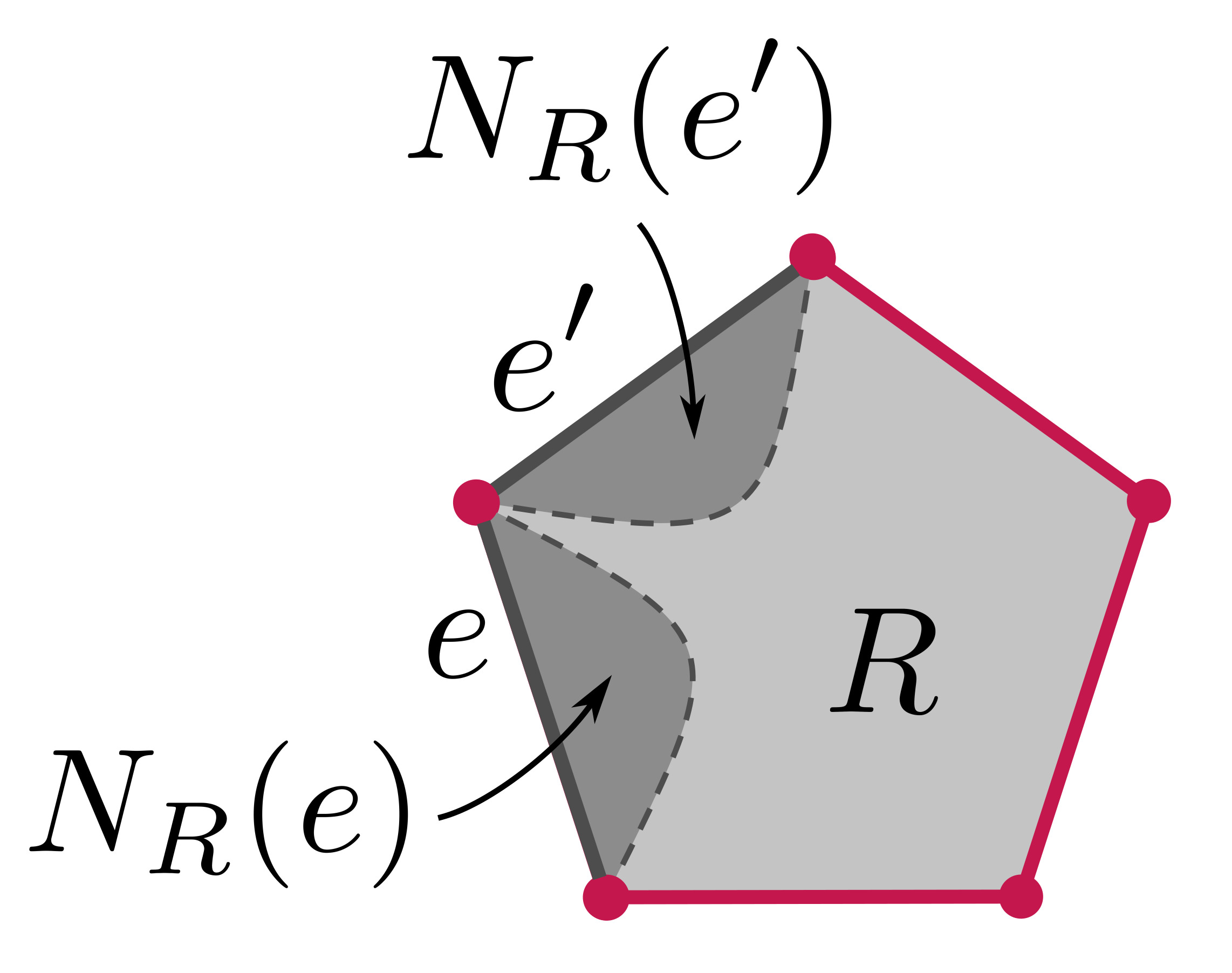} \hfill
            \includegraphics[width=0.38\textwidth]{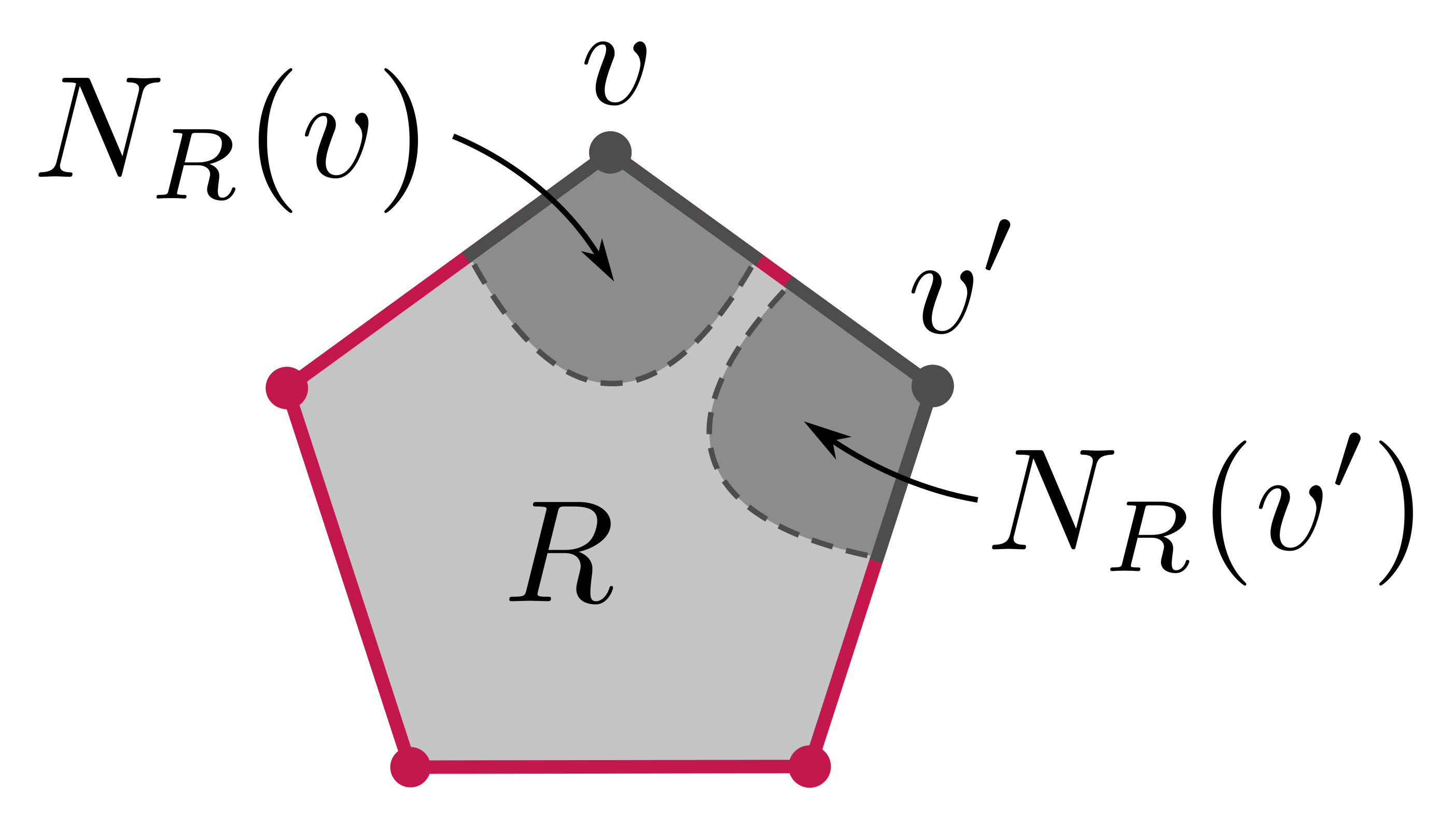}
            \hfill 
        }
        \caption{Left: Two disjoint open sets $N_R(e)$ and $N_R(e')$ of a polygon $R$ that contain $\operatorname{Int}(e)$ and $\operatorname{Int}(e')$, respectively. Right: Two disjoint open sets $N_R(v)$ and $N_R(v')$ that contain $v$ and $v'$, respectively.}
        \label{fig:special_neighborhoods}
    \end{figure}

    Before proving (ii), first, we introduce a notation. Given a Euclidean polygon $R$ and an edge $e$ of $R$, let $N_R(e)$ denote an open subset of $R$ that contains the edge $e$ but does not contain the end points of $e$. Furthermore, the open neighborhoods $N_R(e)$ are chosen such that $N_R(e) \cap N_R(e') = \emptyset$ for all pairs of edges $e \ne e'$ (see Figure \ref{fig:special_neighborhoods}). Now, fix an edge $e$ of the polygonal surface $E$. By definition, the edge $e$ is obtained by gluing exactly two edges $e_1$ and $e_2$ of polygons $P_{\alpha_1}, P_{\alpha_2}$, respectively. Let $N_{P_{\alpha_1}}(e_1)$ and $N_{P_{\alpha_2}}(e_2)$ be the open neighborhoods of $e_1$ and $e_2$. Note that $N_{P_{\alpha_1}}(e_1) \cup N_{P_{\alpha_2}}(e_2)$ is a simply-connected neighborhood of $E$ containing $\operatorname{Int}(e)$. Choose a connected component $U$ of $q^{-1}(N_{P_{\alpha_1}}(e_1) \cup N_{P_{\alpha_2}}(e_2))$. The open set $U$ is contained in two tiles of $S$, which we call $T_{\beta_1}$ and $T_{\beta_2}$ (it is possible that $T_{\beta_1}= T_{\beta_2}$). Then, note that $f^{-1}(U)$ is exactly the union $N_{Q_{\beta_1}}(e'_1) \cup  N_{Q_{\beta_2}}(e'_2)$ for two distinct edges $e'_1$ and $e'_2$. In conclusion, we get that an edge $e_1'$ of $Q_{\beta_1}$ is identified with exactly one other edge $e_2'$ of some polygon $Q_{\beta_2}$.
    
    For (iii), a similar analysis is required. Notation: given a Euclidean polygon $R$ and a vertex $v$ of $R$, let $N_R(v)$ denote an open subset of $R$ that contains the vertex $v$ such that $N_R(v) \cap N_R(v') = \emptyset$ for all pairs of vertices $v \ne v'$. Now, fix a vertex $v$ of the polygonal surface $E$. Let $P_{\alpha_1}, P_{\alpha_2}, \dots, P_{\alpha_d}$ be the cyclic arrangement of polygons around $v$ in $E$. Then, the vertex $v$ is obtained by gluing vertices $v_1,v_2, \dots, v_d$ of polygons $P_{\alpha_1}, P_{\alpha_2}, \dots, P_{\alpha_d}$. Let $N_{P_{\alpha_1}}(v_1), N_{P_{\alpha_2}}(v_2), \dots, N_{P_{\alpha_d}}(v_d)$ denote open neighborhoods of $v_1,v_2, \dots, v_d$. Note that $N_{P_{\alpha_1}}(v_1) \cup N_{P_{\alpha_2}}(v_2) \cup \cdots \cup N_{P_{\alpha_d}}(v_d)$ is a simply-connected neighborhood of $v$ in $E$. Choose a connected component $U$ of $q^{-1}(N_{P_{\alpha_1}}(v_1) \cup N_{P_{\alpha_2}}(v_2) \cup \cdots \cup N_{P_{\alpha_d}}(v_d))$. The open set $U$ is contained in $d$ tiles of $S$, which we call $T_{\beta_1}, T_{\beta_2}, \dots , T_{\beta_d}$ (It is possible that $T_{\beta_i}= T_{\beta_j}$). Then, note that $f^{-1}(U)$ is exactly the union $N_{Q_{\beta_1}}(v'_1) \cup  N_{Q_{\beta_2}}(v'_2) \cup  \cdots \cup N_{Q_{\beta_d}}(v'_d)$ for $d$ distinct vertices $v'_1, v'_2, \dots , v'_d$. In conclusion, two vertices $v'_i$ and $v'_j$ of $\{Q_\beta\}_{\beta \in J}$ are glued together if and only if there exists a sequence of polygons $Q_i, Q_{i+1}, \dots, Q_{j}$ that are glued along edges in a cyclic fashion so as to induce the gluing of vertices $v'_i$ and $v'_j$. In particular, to obtain the quotient surface $( \bigcup_{\beta \in J} Q_\beta )/f$, it suffices to just glue the edges of $\{Q_\beta\}$ as prescribed by the map $f$. Thus, $( \bigcup_{\beta \in J} Q_\beta )/f$ is a polygonal surface $\widetilde E$ and the map $f \colon \bigcup_{\beta \in J} Q_\beta \rightarrow S$ factors to give a bijective map $f \colon \widetilde E \rightarrow S$. This map is a homeomorphism because it is a bijective open map. Therefore, $\pi := q \circ f \colon \widetilde E \rightarrow E$ is a universal covering map.

    For (iv), given a polygon $Q_\beta$ of $\widetilde E$, recall that $\pi(Q_\beta)=P_{\alpha(\beta)}$ and $Q_\beta$ is isometric to $P_{\alpha(\beta)}$. This shows that $n$-gons are mapped to $n$-gons under $\pi$. Similarly, using the discussion above we can conclude that the vertices and edges of $\widetilde E$ are mapped to vertices and edges of $E$, respectively. Lastly, given a vertex $\tilde v \in \pi^{-1}(v)$ of $\widetilde E$, by the discussion above, we have a homeomorphism of simply connected neighborhoods $\pi \colon N_{Q_{\beta_1}}(v'_1) \cup  N_{Q_{\beta_2}}(v'_2) \cup  \cdots \cup N_{Q_{\beta_d}}(v'_d) \rightarrow N_{P_{\alpha(\beta_1)}}(v_1)\cup N_{P_{\alpha(\beta_2)}}(v_2)\cup \dots\cup N_{P_{\alpha(\beta_d)}}(v_d)$, where $\{Q_{\beta_i}\}_{1\leq i \leq d}$ and $\{P_{\alpha(\beta_i)}\}_{1\leq i \leq d}$ is the cyclic arrangement of polygons around the vertex $\tilde v$ and $v$, respectively. This shows that the vertex-type of the vertices $\tilde v$ and $\pi(\tilde v) = v$ are equal.
\end{proof}

\section{Proof of Theorem 1(a) ($\kappa(v)>0$ case)}\label{sec:positive_case}
\begin{Theorem}[Bonnet-Myers for regular polygonal surfaces]\label{thm:Theorem_polygonal_Bonnet_Myers}
    Let $E$ be a regular polygonal surface. Suppose the following conditions hold.
    \begin{enumerate}
        \item[(a)] There is a constant $N$ such that each polygon in $E$ has at most $N$ sides.
        \item[(b)] The combinatorial curvature $\kappa(v)$ at each vertex of $E$ is strictly positive.
    \end{enumerate}
    Then $E$ is compact.
\end{Theorem}
\begin{Remark}
    In Riemannian geometry, the Bonnet-Myers theorem also gives a bound on the diameter in terms of the curvature. Analogously, given a polygonal surface $E$ satisfying the hypothesis of Theorem \ref{thm:Theorem_polygonal_Bonnet_Myers} and a number $c_0 > 0$ such that $\kappa(v)\geq c_0$ for all vertices $v$, then $E$ has at most $2/c_0$ vertices in it. This follows from the discrete analog of the Gauss-Bonnet theorem (see Lemma \ref{thm:Lemma_Gauss_Bonnet}). In addition to this bound, the recent work of Ghidelli \cite{ghidelli2023largest} shows that, other than four special infinite families (the prisms, the antiprisms and their projective analogs), any polygonal surface with strictly positive combinatorial curvature and a degree of at least $3$ at each vertex has at most 208 vertices. Furthermore, this bound is achieved by a polygonal surface with $208$ vertices. (See also \cite{devos2007analogue,nicholson2011new,oldridge2017characterizing,oh2017number}.)
\end{Remark}
Theorem \hyperref[thm:Theorem_regular_polygonal_type]{1(a)} follows as a corollary of Theorem \ref{thm:Theorem_polygonal_Bonnet_Myers} stated above. To see this, let $E$ be an orientable regular polygonal surface such that the hypothesis of Theorem \hyperref[thm:Theorem_regular_polygonal_type]{1(a)} is satisfied, that is, conditions (a) and (b) of Theorem \ref{thm:Theorem_polygonal_Bonnet_Myers} above are satisfied. Using Lemma \ref{thm:Lemma_lift_polygonal_surface}, we obtain a simply-connected regular polygonal surface $\widetilde E$ which is a universal cover of $E$ along with a universal covering map $\pi \colon \widetilde E \rightarrow E$. Additionally, Lemma \ref{thm:Lemma_lift_polygonal_surface} tells us that (i) combinatorial curvature $\kappa(\tilde v)$ at a vertex $\tilde v$ of $\widetilde E$ is equal to the combinatorial curvature $\kappa(\pi(\tilde v))$ at the vertex $\pi(\tilde v)$ of $E$ and (ii) the number of sides in a polygon $\tilde P$ of $\widetilde E$ is equal to the number of sides in the polygon $\pi(\tilde P)$ of $E$. Hence, the universal cover $\widetilde E$ also satisfies hypothesis (a) and (b) of Theorem \ref{thm:Theorem_polygonal_Bonnet_Myers} above and we deduce that $\widetilde E$ is compact. As $\widetilde E$ is simply-connected, it follows that $\widetilde E = \hat{\mathbb{C}}$. Moreover, $E = \widetilde E = \hat{\mathbb{C}}$ because the Riemann sphere $\hat{\mathbb{C}}$ has no non-trivial quotients. This proves Theorem \hyperref[thm:Theorem_regular_polygonal_type]{1(a)} and it remains to prove Theorem \ref{thm:Theorem_polygonal_Bonnet_Myers}.

We begin by providing a two-point summary of the proof.
\begin{enumerate}
    \item A regular polygonal surface $E$ is flat away from the vertices and all the curvature of $E$ is accumulated at the vertices. We ``distribute" the positive curvature at the vertices into the complementary subset $E \setminus V$, where $V$ is the set of vertices of $E$. More specifically, we replace $E$ with a new surface $E_r$ that is homeomorphic to $E$ but has a different geometry: the open subset $E_r \setminus V$ is a smooth Riemannian $2$-manifold and has a constant curvature of $+1/r^2$. Furthermore, $r$ is chosen to be sufficiently large such that the $r$-spherical angle-sum at each vertex of $E_r$ is strictly lesser than $2\pi$. 
    \item We apply the Bonnet-Myers theorem of Riemannian geometry on $E_r \setminus V$ and conclude that the diameter of $E_r \setminus V$ is at most $\pi r$. The bound on the diameter forces $E_r$ and hence $E$ to be compact surfaces.
\end{enumerate}

\medskip
\noindent \textbf{Organization of the section.}
Subsection \ref{subsec:spherical_surface} defines a new family of metric surfaces (regular $r$-spherical surfaces). Next, in Subsection \ref{subsec:spherical_bonnet_myers}, a Bonnet-Myers result is established for the newly defined regular $r$-spherical surfaces. Finally, in Subsection \ref{subsec:culmination_bonnet_myers}, we establish Theorem \ref{thm:Theorem_polygonal_Bonnet_Myers} by relating regular polygonal surfaces with regular $r$-spherical surfaces.

\subsection{Regular $r$-spherical surfaces and their $r$-spherical angle-sums}\label{subsec:spherical_surface}
\begin{figure}[t!]
    \centering
    {
        \hfill
        \includegraphics[width=0.36\textwidth]{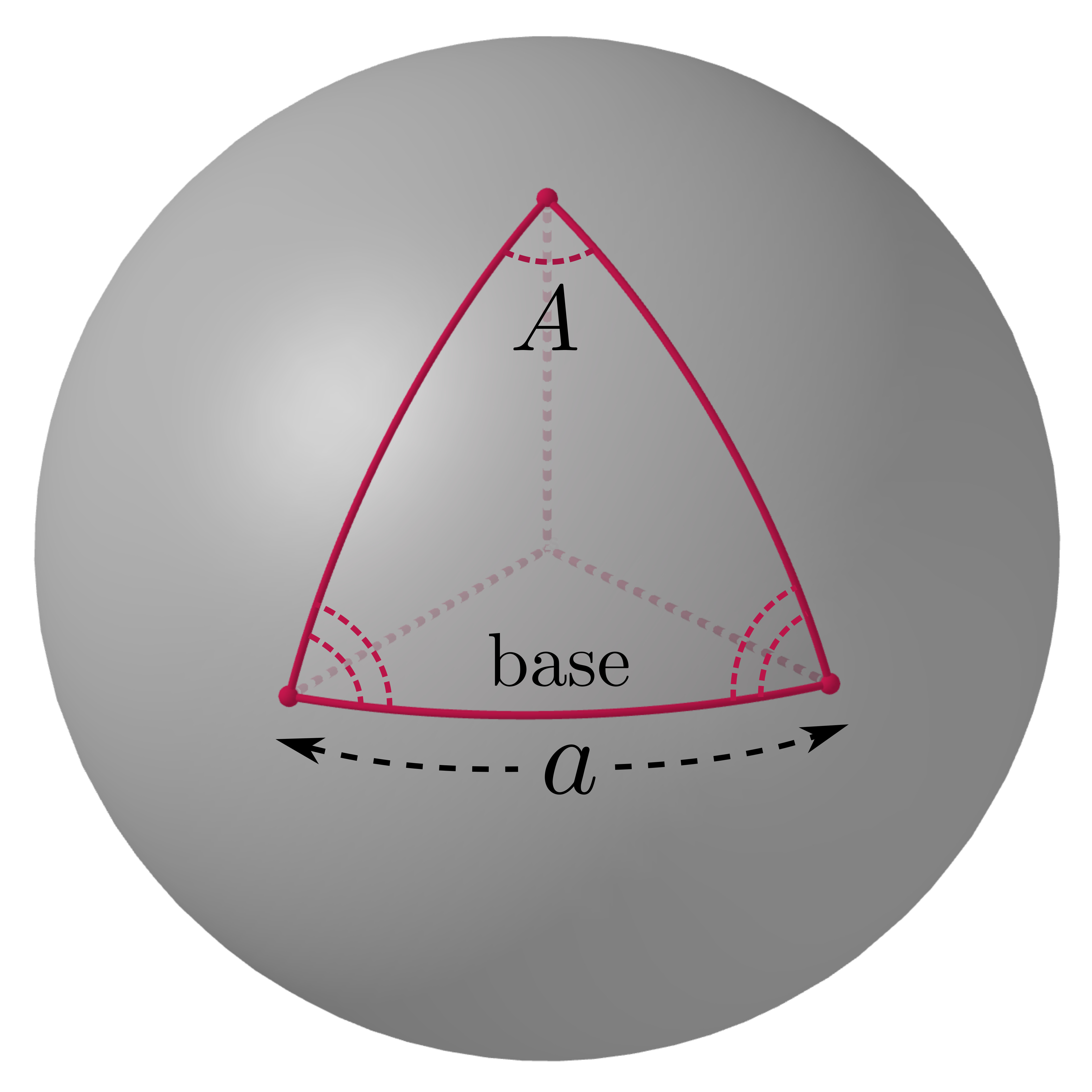} \hfill \hfill
        \includegraphics[width=0.36\textwidth]{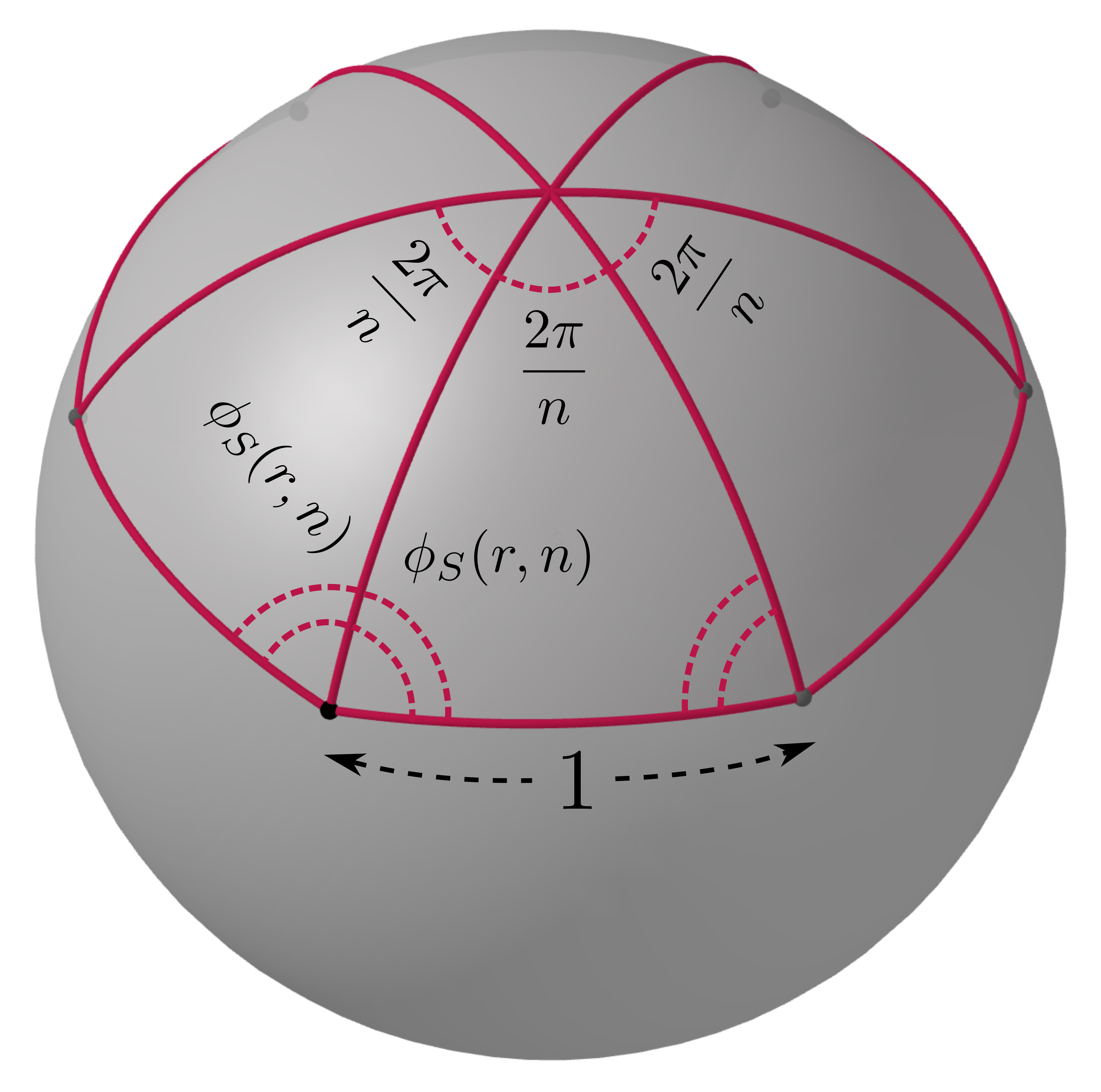}
        \hfill 
    }
    \caption{Left: A spherical isosceles triangle with base length $a$ and opposite angle $A$. Right: A unit regular spherical n-gon constructed as the union of $n$ spherical isosceles triangle with base length $1$ and opposite angle $2\pi/n$.}
    \label{fig:spherical_polygon}
\end{figure}
For each $r \in (0,\infty)$, let $S^2_r\subseteq \mathbb R^3$ be the 2-sphere of radius $r$ endowed with the induced Riemannian metric from $\mathbb R^3$. An $r$-spherical triangle is a closed subset of $S^2_r$ bounded by three geodesic segments. We would like to define a unit regular $r$-spherical $n$-gon using certain $r$-spherical isosceles triangles (see Figure \ref{fig:spherical_polygon}). Towards this, note that, given an angle $0<A<\pi$ and a side-length $0<a<Ar$, there exists a unique (up to isometry) $r$-spherical isosceles triangle such that (i) the base length $a$, (ii) the angle opposite to the base is $A$, and (iii) two base angles are acute angles, that is, they are less than $\pi/2$. Now, given $r>n/(2\pi)$, a \textit{unit regular $r$-spherical $n$-gon} is defined to be the union of $n$ many identical $r$-spherical isosceles triangles $\{T_i\}_{1\leq i \leq n}$ arranged cyclically around a common vertex; where each $T_i$ is a $r$-spherical isosceles triangle with base length $a=1$ and an angle $A=2\pi/n$ opposite to the base. Each $T_i$ has two equal acute angles at the two vertices of the base. This angle, call $\phi_S(r,n)$, can be computed using spherical trigonometry: it is the unique solution of the following equation in the range $(0,\pi)$.
\begin{equation}\label{eqn:spherical_angle-sum}
    \sin^2(\phi_S(r,n)) = \frac{1+\cos(2\pi /n)}{1+\cos(1/r)}
\end{equation}
We make two observations: (i) the interior angle of a unit regular $r$-spherical $n$-gon is $2\phi_S(r,n)$ (see Figure \ref{fig:spherical_polygon}), (ii) in the limit $r \rightarrow \infty$, the interior angle of a unit regular $r$-spherical $n$-gon tends to the interior angle of a unit regular Euclidean $n$-gon, that is,
\begin{equation*}
    \lim_{r \rightarrow \infty} 2\phi_S(r,n) = \pi - \frac{2\pi}{n}.
\end{equation*}
This can be seen using equation \ref{eqn:spherical_angle-sum}.

\begin{Definition}\label{definition:regular_spherical_surface}
    A \textit{regular $r$-spherical surface} $E_r$ is a surface obtained by gluing a collection of unit regular $r$-spherical polygons along their edges such that each edge is identified with exactly one other edge using a spherical isometry.
\end{Definition}

Given a vertex $v$ of $E_r$, suppose $P_{1}, P_{2}, \dots, P_{d}$ is the cyclic arrangement of $r$-spherical polygons around $v$. Then, the vertex-type of $v$ is $[k_1,k_2,\dots,k_d]$, where $k_i$ is the number of sides in the polygon $P_i$. Now, the \textit{$r$-spherical angle-sum} $\mathcal A_r(v)$ at the vertex $v$ is defined as follows:
\begin{equation*}
    \mathcal A_r(v) := \sum_{i=1}^{d} \left( \text{Interior angle of }P_i \right) = \sum_{i=1}^{d} 2\phi_S(r,k_i).
\end{equation*}
For future use, note that, in the limit $r\rightarrow \infty$, the $r$-spherical angle-sum $\mathcal A_r(v)$ tends to the (Euclidean) angle-sum $\mathcal{A}(v)$, that is, 
\begin{equation*}
    \lim_{r \rightarrow \infty} \mathcal A_r(v) =  \lim_{r \rightarrow \infty} \sum_{i=1}^{d} 2\phi_S(r,k_i) = \sum_{i=1}^{d} \left( \pi - \frac{2\pi}{k_i} \right)= \mathcal A(v).
\end{equation*}

\subsection{A Bonnet-Myers theorem for regular $r$-spherical surfaces} \label{subsec:spherical_bonnet_myers}
In this subsection, we prove an analog of Theorem \ref{thm:Theorem_polygonal_Bonnet_Myers} for regular $r$-spherical surfaces.
\begin{Theorem}[Bonnet-Myers for regular $r$-spherical surfaces]\label{thm:Theorem_spherical_Bonnet-Myers}
    Let $E_r$ be a regular $r$-spherical surface where the $r$-spherical angle-sum $\mathcal{A}_r(v)$ is strictly less than $2\pi$ at each vertex $v$. Then, the diameter of $E_r$ is lesser than or equal to $\pi r$ and $E_r$ is compact.
\end{Theorem}
The following proposition plays a key role in the proof of Theorem \ref{thm:Theorem_spherical_Bonnet-Myers} above.
\begin{Proposition}[$E_r \setminus V$ is convex]\label{thm:Proposition_length-minizing_misses_vertices}
    Suppose $E_r$ is a regular $r$-spherical surface where the $r$-spherical angle-sum $\mathcal{A}_r(v)$ is strictly less than $2\pi$ at each vertex $v$. If $\gamma$ is a minimizing piece-wise geodesic joining $x$ to $y$, then the image of $\gamma$ in $E_r$ does not intersect any vertex of $E_r$, except possibly the endpoints $x$ and $y$.
\end{Proposition}
We postpone the proof of Proposition \ref{thm:Proposition_length-minizing_misses_vertices} to the end of this subsection. Now, we introduce piece-wise geodesics and prove Theorem \ref{thm:Theorem_spherical_Bonnet-Myers} using Proposition \ref{thm:Proposition_length-minizing_misses_vertices} above. 

A path $\gamma \colon [0,1] \rightarrow E_r$ is called a \textit{piece-wise geodesic} if there is a partition of $[0,l]$ given by $0 = a_1 < a_2 < \cdots < a_{m+1} = l$ such that each piece $\gamma([a_j,a_{j+1}])$ lies inside a single $r$-spherical polygon $P_{\beta_j}$ and $\gamma|_{[a_j,a_{j+1}]}$ is a geodesic segment of $P_{\beta_j} \subseteq S^2_r$ in the Riemannian metric of $S^2_r$. Each such segment $\gamma|_{[a_j,a_{j+1}]}$ has a well-defined length $\len(\gamma|_{[a_j,a_{j+1}]})$ as measured in the Riemannian metric of $P_{\beta_j} \subseteq S^2_r$. Using this, the \textit{length} of a piece-wise geodesic $\gamma$ is defined to be $\len(\gamma):=\sum_{j=1}^m \len(\gamma|_{[a_j,a_{j+1}]})$. We caution that a generic piece-wise geodesic $\gamma$ from $x$ to $y$ may not be the shortest curve joining $x$ to $y$. 
\begin{Definition}\label{thm:Definition_spherical_pseudo_metric}
    A pseudo-metric $d \colon E_r \times E_r \rightarrow \mathbb [0,\infty)$ on $E_r$ is defined as follows:
    \begin{equation}\label{eqn:spherical_metric}
        d(x,y) := \inf \{\len(\gamma) \colon \gamma \text{ is a piece-wise geodesic from } x \text{ to } y \}.
    \end{equation}
\end{Definition}
The infimum is always finite because any two points in $E_r$ can be joined by a piece-wise geodesic. Further, the function $d$ is a pseudo-metric because it satisfies (i) $d(x,x)=0$ for all $x \in X$, (ii) the symmetry condition: $d(x,y) = d(y,x)$ for all $x,y \in X$, and (iii) the triangle inequality.
\begin{Lemma}[Properties of the metric $d$ on $E_r$]\label{thm:Lemma_spherical_metric}
    Let $d \colon E_r \times E_r \rightarrow \mathbb [0,\infty)$ be the pseudo-metric defined in equation \ref{eqn:spherical_metric}. Then, $d$ has the following properties:
    \begin{enumerate}
        \item[(a)] the function $d$ is a metric on $E_r$;
        \item[(b)] for all points $x,y \in E_r$, there is a \textit{minimizing piece-wise geodesic} joining $x$ and $y$; this means that there is a piece-wise geodesic from $x$ to $y$ whose length is exactly $d(x,y)$;
        \item[(c)] every closed and bounded set of $E_r$ is compact;
        \item[(d)] the open set $E_r \setminus V$ is a smooth connected surface with a canonical Riemannian metric $g$ of constant curvature $+1/r^2$. Further, the distance function $d_g$ induced by $g$ is equal to the distance function $d$ restricted to $E_r\setminus V$. 
    \end{enumerate}
\end{Lemma}
We defer the proof of Lemma \ref{thm:Lemma_spherical_metric} to Appendix \ref{Appendix_metric} and proceed to discuss the Bonnet-Myers theorem of Riemannian geometry.

\medskip
\noindent \textbf{Bonnet-Myers Theorem (convex version).}\label{thm:Proposition_modified_Bonnet_Myer}
\textit{Let $M$ be a connected Riemannian manifold all of whose sectional curvatures are bounded below by the positive constant $1/r^2$. Further, suppose that any two points $x,y$ in $M$ can be joined by a minimizing geodesic. Then, $M$ has a diameter less than or equal to $\pi r$.}
\begin{Remark}
    The usual hypothesis in the Bonnet-Myers theorem is that $M$ is a complete manifold. But on examining the proof of the theorem, we realize that the assumption about completeness is only used to show that $M$ is convex, that is, any two points can be joined by a minimizing geodesic. Thus, we can restrict the hypothesis to the convexity of $M$. As an example, note that the convex version of the Bonnet-Myers theorem stated above can be applied to a convex open ball $B(p,\delta)$ of the sphere $S^2$ to deduce that the diameter of the ball is less than $\pi r$. The usual version cannot be applied because the open ball $B(p,\delta)$ is not complete.
\end{Remark}
For a proof of the Bonnet-Myers Theorem, see, for example, Theorem 11.7 of \cite[pg. 200]{lee2006riemannian} or Theorem 3.1 of \cite[pg. 200]{do1992riemannian}. 
\begin{proof}[Proof of Theorem \ref{thm:Theorem_spherical_Bonnet-Myers}]
    Let $E_r$ be a regular $r$-spherical surface with $r$-spherical angle-sum strictly less than $2\pi$ at each vertex. Using Lemma \hyperref[thm:Lemma_spherical_metric]{3.4(b,d)} along with Proposition \ref{thm:Proposition_length-minizing_misses_vertices}, we deduce that $E_r \setminus V$ is a connected and convex Riemannian manifold of constant curvature $+1/r^2$. Then, we use the Bonnet-Myers theorem (convex version) to deduce that the diameter of $E_r$ is at most $\pi r$. As the closure of $E_r \setminus V$ is $E_r$, the diameter of $E_r$ is also lesser than or equal to $\pi r$. Lastly, as $E_r$ is closed and bounded, we deduce that $E_r$ is compact by Lemma \hyperref[thm:Lemma_spherical_metric]{3.4(c)}.
\end{proof}

We conclude this subsection with the proof of Proposition \ref{thm:Proposition_length-minizing_misses_vertices}.

\begin{proof}[Proof of Proposition \ref{thm:Proposition_length-minizing_misses_vertices}]
    We shall prove the contrapositive. Let $E_r$ be a regular $r$-spherical surface with $r$-spherical angle-sum strictly less than $2\pi$ at each vertex. Given a piece-wise geodesic $\gamma$ from $x$ to $y$, suppose that the interior of $\gamma$ contains a vertex $v$ of $E_r$, that is, $\gamma(t)=v$ for some $0<t<1$. Then, we shall show that $\operatorname{len}(\gamma)>d(x,y)$.

    Let $P_{1},P_{2},\dots, P_{d}$ be the cyclic arrangement of $r$-spherical polygons around the vertex $v$ in $E_r$. As $\gamma$ is a piece-wise geodesic, there exists $0<a<t<b<1$ such that $\gamma|_{[a,t]}$ and $\gamma|_{[t,b]}$ are geodesic segments in some $r$-spherical polygons $P_{i_a}$ and $P_{i_b}$, respectively. We leave the case of $i_a=i_b$ for the reader (see Figure \ref{fig:geodesic_segment_simple}).

    \begin{figure}[t!]
        \centering
        {
            \hfill
            \includegraphics[width=0.36\textwidth]{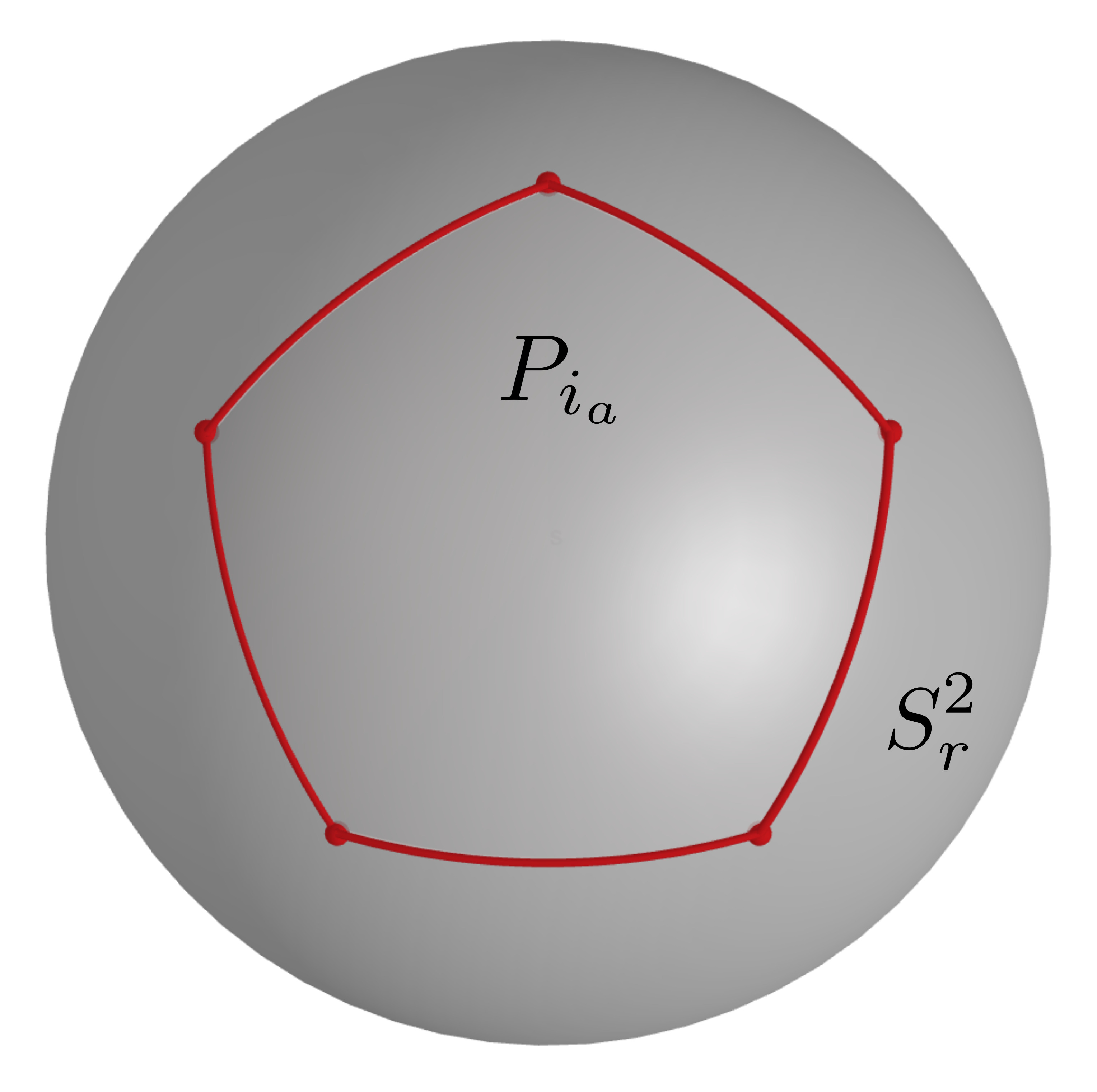} \hfill \hfill
            \includegraphics[width=0.36\textwidth]{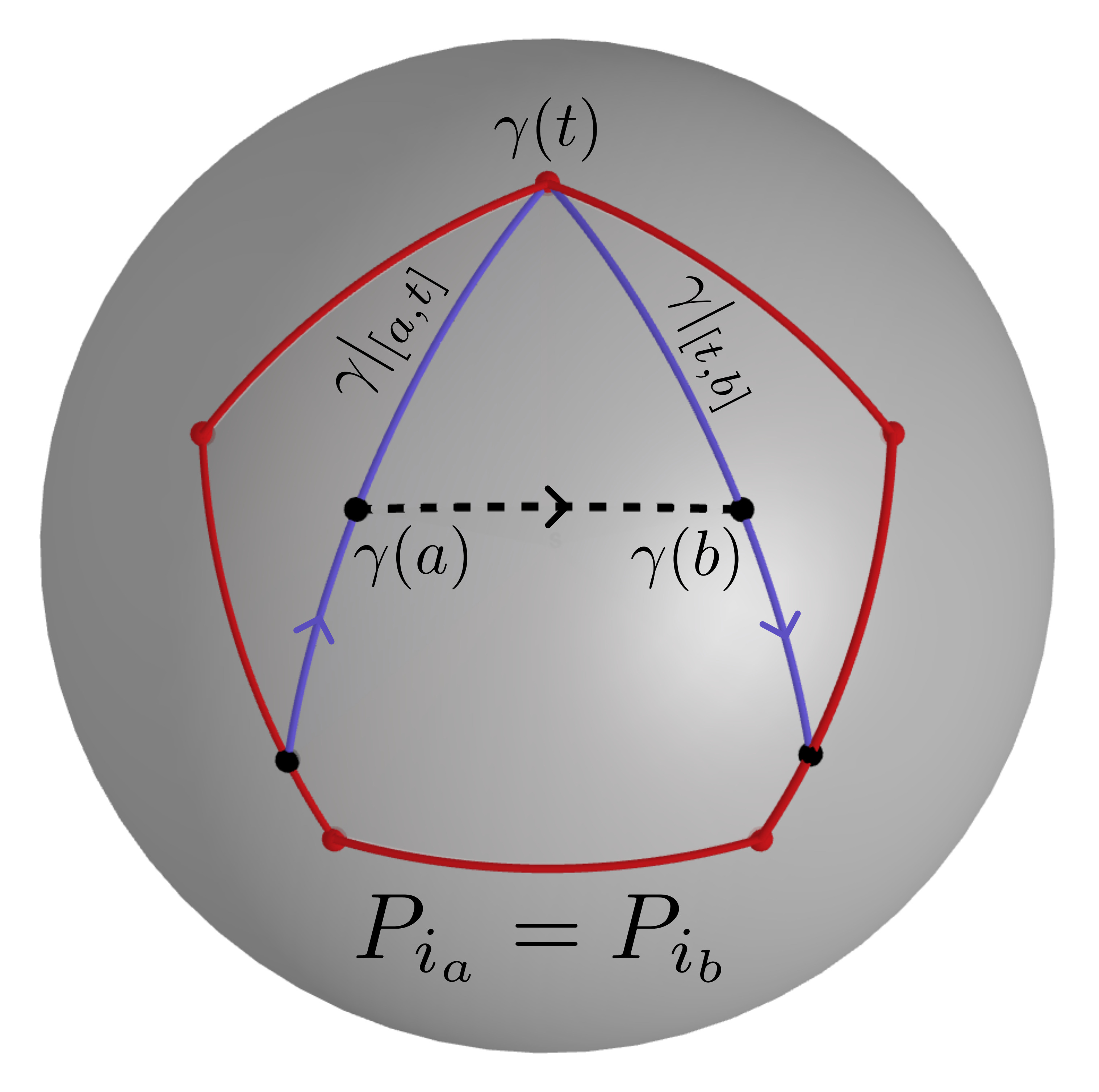}
            \hfill 
        }
        \caption{Left: A spherical polygon $P_{i_a}$ embedded in the sphere $S^2_r$. Right: A piece-wise geodesic $\gamma$ in $P_{i_a}$ that is not length-minimizing.}
        \label{fig:geodesic_segment_simple}
    \end{figure}
    
    Now suppose $i_a \neq i_b$ and observe that $\gamma|_{[a,t]} \cup \gamma|_{[t,b]}$ cuts a neighborhood of the vertex $v$ into two ``sectors" (see Figure \ref{fig:geodesic_segment}). Moreover, we can define the angles $\theta_1,\theta_2$ of the two sectors as follows. Define 
    \begin{equation*}
        \theta_1 := \theta(\gamma|_{[a,t]},e_{i_a}) + \sum_{k=i_a + 1}^{i_b -1} 2\phi_S(r,|P_k|)  + \theta(e_{i_b-1}, \gamma|_{[t,b]}),
    \end{equation*}
    where $e_{i_a}$ is the common edge between $P_{i_a}$ and $P_{i_a+1}$ containing the vertex $v$ as one of its vertices; $\theta(\gamma|_{[a,t]},e_{i_a})$ is the acute angle between $\gamma|_{[a,t]}$ and $e_{i_a}$ at the vertex $v$ in $P_{i_a}$; the angle $2\phi_S(r,|P_k|)$ is the interior angle of the polygon $P_{k}$ at the vertex $v$; and $\theta(e_{i_b-1}, \gamma|_{[t,b]})$ is the acute angle between $e_{i_b-1}$ and $\gamma|_{[t,b]}$ at the vertex $v$ in $P_{i_b}$. Similarly, define the complementary angle as 
    \begin{equation*}
        \theta_2 := \theta(\gamma|_{[t,b]},e_{i_b}) + \sum_{k=i_b + 1}^{i_a -1} 2\phi_S(r,|P_k|) + \theta(e_{i_a-1}, \gamma|_{[a,t]}).
    \end{equation*}
    The two angles $\theta_1$ and $\theta_2$ are complementary in the sense that $\theta_1 + \theta_2 = \mathcal{A}_r(v)$. Next, recall that the angle-sum $\mathcal{A}_r(v)$ is strictly lesser than $2\pi$. As $\theta_1 + \theta_2 = \mathcal{A}_r(v)$, it follows that one of the angles $\theta_i$ is strictly lesser than $\pi$. Now, Proposition \ref{thm:Proposition_length-minizing_misses_vertices} boils down to the following claim.
    \begin{Claim} 
        If one of the angles $\theta_i$ is strictly lesser than $\pi$, then the piece-wise geodesic $\gamma$ is not a minimizing piece-wise geodesic (see Figure \ref{fig:geodesic_segment}).
    \end{Claim}

    \begin{figure}[t!]
        \centering
        {
            \hfill
            \includegraphics[width=0.39\textwidth]{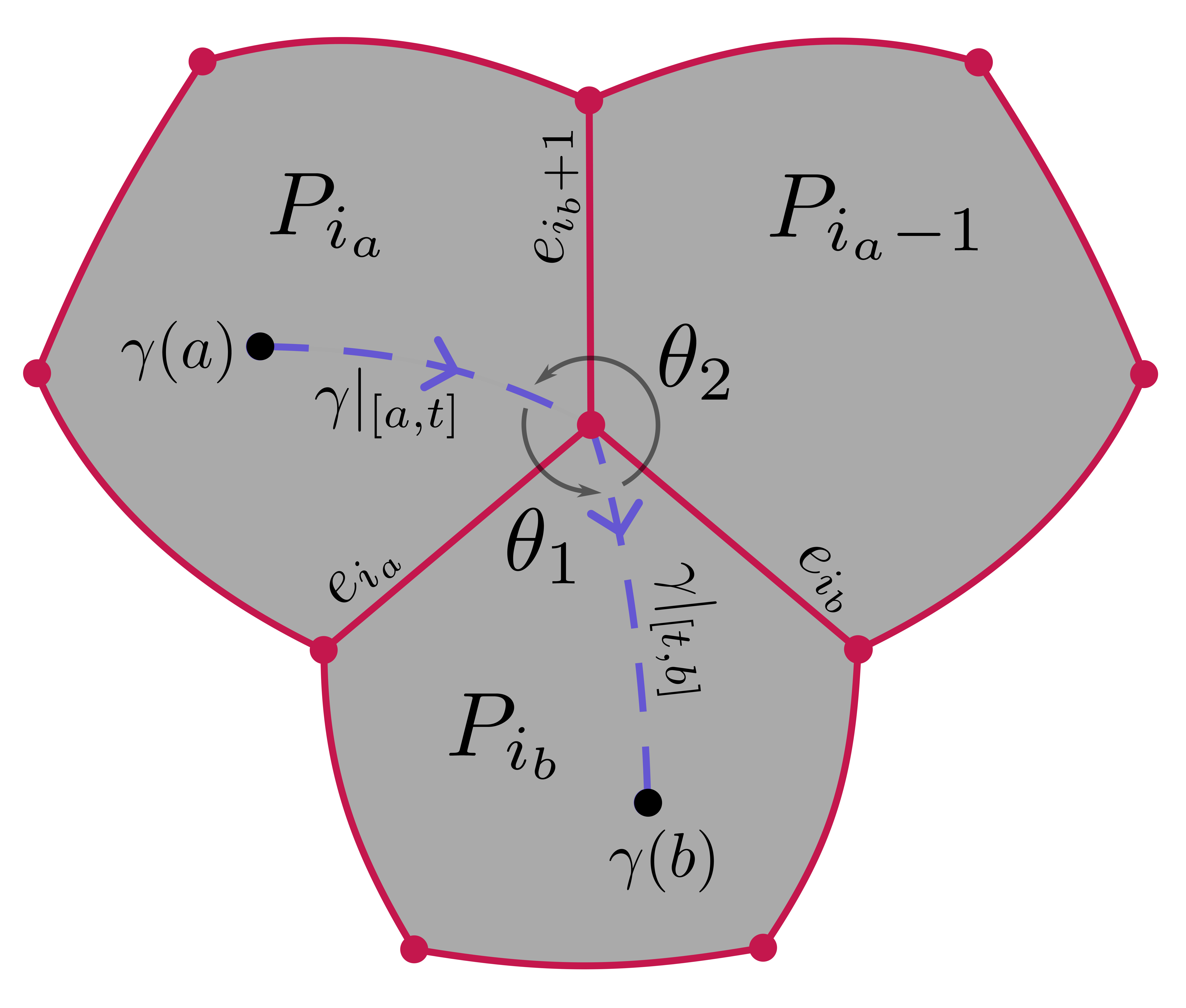} \hfill \hfill
            \includegraphics[width=0.36\textwidth]{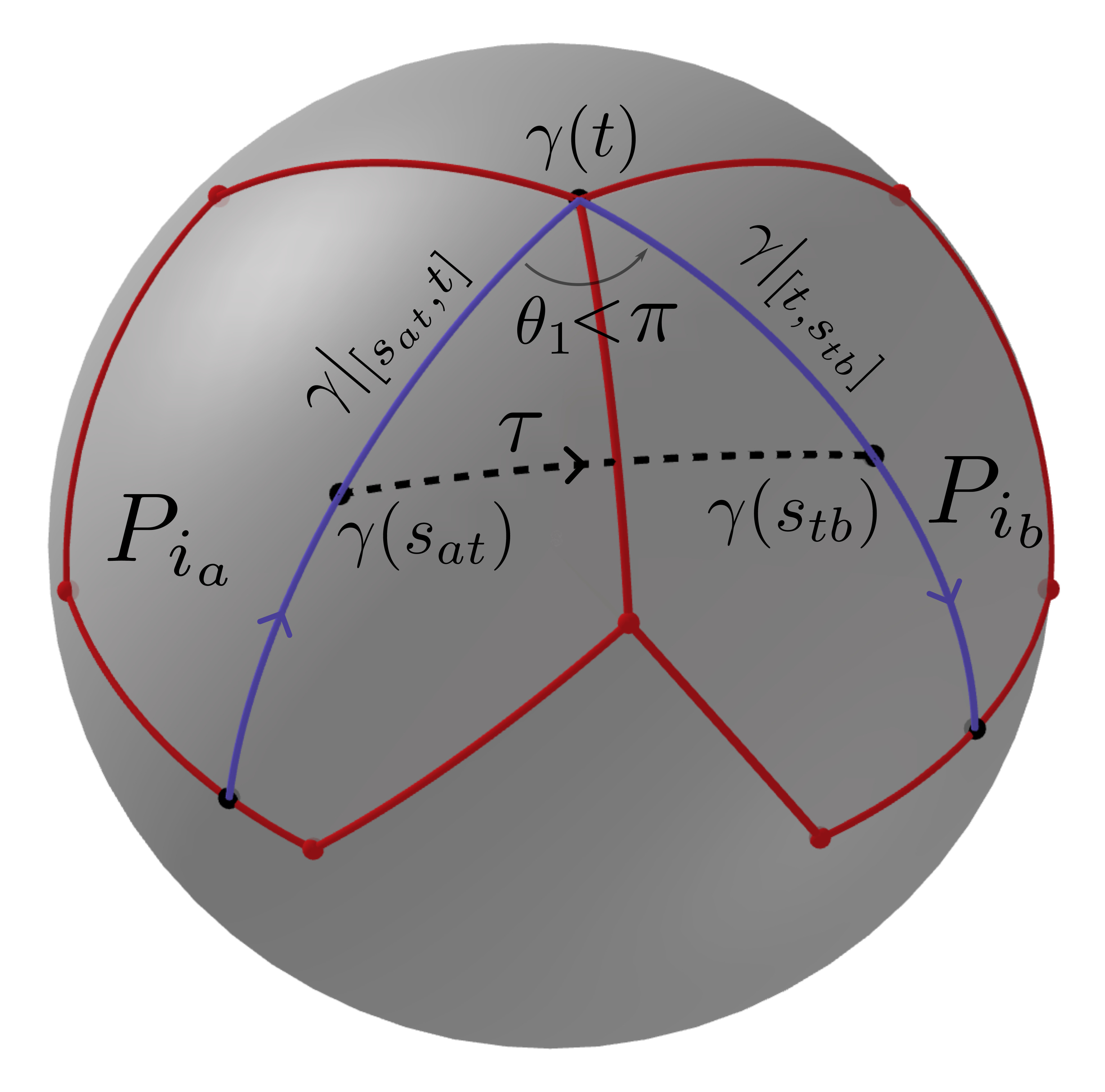}
            \hfill 
        }
        \caption{Left: Two complementary angles $\theta_1$ and $\theta_2$ around a piece-wise geodesic $\gamma$ passing through a vertex. Right: If $\theta_1<\pi$, then $\gamma$ is not length-minimizing because there is a shorter path using $\tau$.}
        \label{fig:geodesic_segment}
    \end{figure}
    To prove the claim we shall show there exists two points $\gamma(s_{at})$ and $\gamma(s_{tb})$ such that $s_{at}\in (a,t)$, $s_{tb}\in(t,b)$, and $d(\gamma(s_{at}), \gamma(s_{tb})) < \len(\gamma|_{[s_{at},t]}) + \len(\gamma|_{[t,s_{tb}]})$. This tells us $\gamma|_{[s_{at},s_{tb}]}$ is not a minimizing piece-wise geodesic. Without loss of generality, we assume that $\theta_1<\pi$. Also, for simplicity of proof, we assume that $i_b \neq i_a - 1$ and the sum $\sum_{k=i_a}^{i_b} 2\phi_S(r,|P_k|)$ is strictly less than $2\pi$. Now, we would like to embed the union $P_{i_a}\cup P_{i_a+1} \cup \cdots \cup P_{i_b}$ into $S^2_r$ in order to better understand the minimizing piece-wise geodesics. 
    
    We achieve the embedding by induction. First, note that the polygons $P_{i}$ are defined as a closed subset of $S^2_r$; we use this to embed them. Start with some embedding of $P_{i_a}$ in $S^2_r$. Next, embed $P_{i_a+1}$ into $S^2_r$ such that the embedding of $P_{i_a}$ and $P_{i_a+1}$ agree on the common edge $e_{i_a}$. This can be achieved by using an isometry of $S^2_r$. Then, we glue the two embeddings together. To show that the glued map is an embedding, it suffices to show that the glued map is injective. The injectivity of the glued map follows because the sum of angles at each of two common vertices of $P_{i_a}$ and $P_{i_a+1}$ is strictly less than $2\pi$. It is here that we use the simplifying assumption that $\sum_{k=i_a}^{i_b} 2\phi_S(r,|P_k|)$ is strictly less than $2\pi$. Continuing in a similar fashion, we embed $P_{i_a+2}$ into $S^2_r$ such that the embedding of $P_{i_a+1}$ and $P_{i_a+2}$ agrees on the common edge $e_{i_a+1}$. Repeating the argument finitely many times, we obtain an embedding of $P_{i_a} \cup P_{i_a+1} \cup \cdots \cup P_{i_b}$ into $S^2_r$.

    Now, we find $\gamma(s_{at})$ and $\gamma(s_{tb})$ such that $d(\gamma(s_{at}), \gamma(s_{tb})) < \len(\gamma|_{[s_{at},t]}) + \len(\gamma|_{[t,s_{tb}]})$. Consider a small $\epsilon$ such that the boundary $\partial B(v,\epsilon)$ of a ball in $S^2_r$ intersects both $\gamma|_{[a,t]}$ and $\gamma|_{[t,b]}$. Let the intersection points be $\gamma(s_{at})$ and $\gamma(s_{tb})$. Note that $a<s_{at}<t$ and $t<s_{tb}<b$. Next, consider the minimizing geodesic $\tau \colon [0,1] \rightarrow S^2_r$ in $S^2_r$ from $\gamma(s_{at})$ to $\gamma(s_{tb})$. By choosing $\epsilon$ sufficiently small, we can ensure that the image of $\tau$ lies completely in the subset $P_{i_a} \cup P_{i_a+1} \cup \cdots \cup P_{i_b}$ embedded in $S^2_r$. Lastly, we show $\tau$ is shorter than $\gamma|_{[s_{at},s_{tb}]}$, that is, $\len(\tau) < \len(\gamma|_{[s_{at},t]}) + \len(\gamma|_{[t,s_{tb}]})$. Note that if the lengths were equal, then $\gamma|_{[s_{at},s_{tb}]}$ would also be a length-minimizing curve between $\gamma(s_{at})$ and $\gamma(s_{tb})$ in $S^2_r$. But a length-minimizing curve in $S^2_r$ is necessarily smooth, while $\gamma|_{[s_{at},s_{tb}]}$ is not smooth at the vertex $v$ as angle at that point is $\theta_1 < \pi$, that is, we have a contradiction. This completes the proof of the claim and Proposition \ref{thm:Proposition_length-minizing_misses_vertices}.
\end{proof}

\subsection{Culmination and conclusion of Theorem \ref{thm:Theorem_polygonal_Bonnet_Myers} (polygonal Bonnet-Myers)}\label{subsec:culmination_bonnet_myers}
\begin{Proposition}\label{thm:Proposition_polygonal_to_spherical}
   Let $E$ be a regular polygonal surface. Suppose the following conditions hold.
    \begin{enumerate}
        \item[(a)] There is a constant $N$ such that each polygon in $E$ has at most $N$ sides.
        \item[(b)] The angle-sum $\mathcal{A}(v)$ at each vertex of $E$ is strictly lesser than $2\pi$.
    \end{enumerate} 
    Then, there is a large $r>0$ and a regular $r$-spherical surface $E_r$, such that $E_r$ is homeomorphic to $E$ and has $r$-spherical angle-sum $\mathcal{A}_r(v)$ strictly less than $2\pi$ at each vertex. 
\end{Proposition}

Theorem \ref{thm:Theorem_polygonal_Bonnet_Myers} is now a corollary of Proposition \ref{thm:Proposition_polygonal_to_spherical} and Theorem \ref{thm:Theorem_spherical_Bonnet-Myers}. To see this, suppose $E$ is a regular polygonal surface that satisfies the hypothesis of Theorem \ref{thm:Theorem_polygonal_Bonnet_Myers}, that is, $E$ satisfies conditions (a) and (b) given above. Using Proposition \ref{thm:Proposition_polygonal_to_spherical}, we obtain a regular $r$-spherical surface $E_r$ that is homeomorphic to $E$. This $r$-spherical surface satisfies the hypothesis of Theorem \ref{thm:Theorem_spherical_Bonnet-Myers} because every $r$-spherical angle-sum in $E_r$ is strictly less than $2\pi$. Hence, Theorem \ref{thm:Theorem_spherical_Bonnet-Myers} implies that $E_r$ and hence $E$ are compact. 

It only remains to prove Proposition \ref{thm:Proposition_polygonal_to_spherical}. We first prove a useful lemma.
\begin{Lemma}[Replacing polygons with spherical polygons]\label{thm:Lemma_polygonal_to_spherical}
    Suppose $E$ is a regular polygonal surface and $N$ is a constant such that each polygon in $E$ has at most $N$ sides. Let $t$ be a number such that $t>N/(2\pi)$. Then, there is a $t$-spherical surface $E_t$ homeomorphic to $E$. Furthermore, the homeomorphism between $E$ and $E_t$ takes vertices, edges, and polygons of $E$ to vertices, edges, and spherical polygons of $E_t$, respectively. In particular, the vertex-type of a vertex $v_{E_t}$ of $E_t$ is the same as the vertex-type of the corresponding vertex $v$ in $E$.
\end{Lemma}
\begin{proof}
    Suppose $E$ is a regular polygonal surface that satisfies the hypothesis. Let $\{P_\alpha\}_{\alpha \in \Lambda}$ be the collection of unit regular Euclidean polygons in $E$ and $t$ be a number such that $t>N/(2\pi)$. As each polygon in $\{P_\alpha\}_{\alpha \in \Lambda}$ has at most $N$ sides, we can consider a collection of unit regular $t$-spherical polygons $\{P_{t,\alpha}\}_{\alpha \in \Lambda}$ where each $P_{t,\alpha}$ has the same number of sides as $P_\alpha$. Now we identify the edges of the spherical polygons $\{P_{t,\alpha}\}_{\alpha \in \Lambda}$ in the same pattern as how the edges of the Euclidean polygons $\{P_\alpha\}_{\alpha \in \Lambda}$ are identified in $E$. This results in a $t$-spherical surface $E_t$. 
    
    Next, we consider a collection of homeomorphisms $\{g_\alpha \colon P_{t,\alpha} \rightarrow P_\alpha\}_{\alpha \in \Lambda}$, where each $g_\alpha$ maps an edge of $P_{t,\alpha}$ to the corresponding edge of $P_\alpha$ by an isometry of segments of unit length. Such a collection of maps $\{g_\alpha\}_{\alpha \in \Lambda}$ glue to give a homeomorphism $g \colon E_t \rightarrow E$. This homeomorphism maps vertices, edges and spherical polygons of $E_t$ to vertices, edges and polygons of $E$, respectively.
\end{proof}
\begin{proof}[Proof of Proposition \ref{thm:Proposition_polygonal_to_spherical}]
    Given a regular polygonal surface $E$ satisfying the hypothesis, let $N$ be the upper bound on the number of sides of the polygons in $E$. Using Lemma \ref{thm:Lemma_polygonal_to_spherical}, for each radius $t>N/(2\pi)$, get a $t$-spherical surface $E_t$ which is homeomorphic to $E$. It remains to show that there exists a $t_0\in(N/(2\pi), \infty)$ such that the $t_0$-spherical surface $E_{t_0}$ has $t_0$-spherical angle-sum strictly lesser than $2\pi$ at each vertex of $E_{t_0}$.

    Given a vertex $v$ of $E$, suppose $[k_1,k_2,\dots,k_d]$ is the vertex-type at the vertex $v$. Next, for each $t >N/(2\pi)$, let $v_{E_t}$ be the vertex in $E_t$ that corresponds to the vertex $v$ in $E$. Recall, the $r$-spherical angle-sum at the vertex $v_{E_t}$ is defined by the following formula.
    \[
    \mathcal A_t(v_{E_t}) = \sum_{i=1}^{d} 2\phi_S(t,k_i)
    \]
    Three observations are in order.
    \begin{enumerate}
        \item In the limit $t \rightarrow \infty$, the $r$-spherical angle-sum $\mathcal A_t(v_{E_t})$ tends to the Euclidean angle-sum $\mathcal{A}(v)$. Then, recall that $\mathcal{A}(v) < 2\pi$ by hypothesis.
        \item The $t$-spherical angle-sum at the vertex $v_{E_t}$ only depends on the radius $t$ and the vertex-type $[k_1,k_2,\dots,k_d]$ at the vertex $v$. In particular, it has no dependence on the surface $E_t$. Thus, we use the notation $\mathcal{A}_t([k_1,k_2,\dots,k_d]):=\mathcal A_t(v_{E_t})$.
        \item There are only finitely many distinct vertex-types in $E_t$. To see this, first note that the degree $d$ at each vertex is at most $6$ because a vertex $v$ with degree more than $7$ will have an angle-sum $\mathcal{A}(v)$ greater than $2\pi$. Next, recall that each polygon in $E$ has at most $N$ sides. This means that the vertex-type of each vertex in $E$ is an element of the finite set $\{[k_1,k_2,\dots,k_d] \colon k_i \leq N \text{ and } d\leq 6\}$.
    \end{enumerate}
    For brevity, we denote a vertex-type $[k_1,k_2,\dots,k_d]$ by \texttt {K}. Let $\texttt{K}_1,\texttt{K}_2,\dots, \texttt{K}_m$ be the finitely many distinct vertex-types in $E$. For a fixed $1\leq j \leq m$, by observation (1) above, we have
    \[ \lim_{t \rightarrow \infty} \mathcal{A}_t(\texttt{K}_j) = \mathcal{A}(\texttt{K}_j).\]
    As the limit $\mathcal{A}(\texttt{K}_j)$ is strictly lesser than $2\pi$, we can choose a $t_0$ such that $\mathcal{A}_{t_0}(\texttt{K}_j) < 2\pi$. Moreover, as there are only finitely many indices $j$, we can choose a sufficiently large $t_0$ such that $\mathcal{A}_{t_0}(\texttt{K}_j) < 2\pi$ for all $1\leq j \leq m$. By observations (2) and (3) above, we have $\mathcal{A}_{t_0}(v_{E_{t_0}}) < 2\pi$ for each vertex $v_{E_{t_0}}$ in $E_{t_0}$. The surface $E_{t_0}$ is the required regular $t_0$-spherical surface that is homeomorphic to $E$.
\end{proof}

\section{Proof of Theorem 1(b) ($\kappa(v) = 0$ case)}\label{sec:zero_case}
\begin{Proposition}\label{thm:Proposition_angle_2pi_regular_polygonal_surface}
    Suppose $\widetilde E$ is a simply-connected oriented polygonal surface. If the angle-sum at each vertex of $\widetilde E$ is equal to $2\pi$, then $\widetilde E$ is conformally isomorphic to $\mathbb C$.
\end{Proposition}
Theorem \hyperref[thm:Theorem_polygonal_Bonnet_Myers]{1(b)} follows as a corollary of Proposition \ref{thm:Proposition_angle_2pi_regular_polygonal_surface} above. To see this, let $E$ be an orientable regular polygonal surface where the combinatorial curvature at each vertex is zero. Using Lemma \ref{thm:Lemma_lift_polygonal_surface}, we obtain a simply-connected regular polygonal surface $\widetilde E$ which is a universal cover of $E$. Additionally, Lemma \ref{thm:Lemma_lift_polygonal_surface} tells us that combinatorial curvature $\kappa(\tilde v)$ at each vertex $\tilde v$ of $\widetilde E$ is zero. For regular polygonal surfaces, recall that angle-sum $\mathcal{A}(w)$ at a vertex $w$ is equal to $2\pi - 2\pi \kappa(w)$. It follows that the angle-sum at each vertex of $\widetilde E$ is exactly $2\pi$. Applying Proposition \ref{thm:Proposition_angle_2pi_regular_polygonal_surface} above on $\widetilde E$, we conclude that $\widetilde E =\mathbb C$ and $E$ is parabolic. This proves Theorem \hyperref[thm:Theorem_regular_polygonal_type]{1(b)}.

The proof of Proposition \ref{thm:Proposition_angle_2pi_regular_polygonal_surface} relies crucially on the following fact from Riemannian geometry. 
\begin{Lemma}\label{thm:Lemma_flat_R^2}
    Suppose $M$ is a $2$-dimensional simply-connected Riemannian manifold. If the metric of $M$ is complete and flat, then $M$ is isometric to $\mathbb R^2$.
\end{Lemma}
For a proof of Lemma \ref{thm:Lemma_flat_R^2}, see, for example, Theorem 4.1 of \cite[pg. 163]{do1992riemannian} or Theorem 11.12 of \cite[pg. 204]{lee2006riemannian}.
\begin{proof}[Proof of Proposition \ref{thm:Proposition_angle_2pi_regular_polygonal_surface}]
    Suppose $\widetilde E$ is a simply-connected polygonal surface with angle-sum $2\pi$ at each vertex. Recall from Lemma \ref{thm:Lemma_flat_conformal_metric} that $\widetilde E$ has a flat conformal metric $\rho$. We shall show the metric $\rho$ is complete and apply Lemma \ref{thm:Lemma_flat_R^2} to obtain an isometry $H \colon \widetilde{E} \rightarrow \mathbb R^2=\mathbb C$. 

    To show the metric $\rho$ is complete, start with the observation that each polygon $P_\alpha$ in $\widetilde E$ is compact. Next, for each $P_\alpha$ in $\widetilde E$ define $P_\alpha^\text{nbd}$ to be the finite union of all the polygons (including $P_\alpha$) that have a non-empty intersection with $P_\alpha$. Now, given a Cauchy sequence $\{a_n\}$ in $\widetilde E$, excluding finitely many terms if necessary, we may assume $d_{\rho}(a_n,a_m)<1/2$ for all $n,m\in \mathbb N$. Each $a_n$ lies in some polygon of $\widetilde E$. Let $a_1$ lie in $P_\alpha$. Observe that the entire sequence $\{a_n\}$ lies in $P_\alpha^\text{nbd}$, where $P_\alpha^\text{nbd}$ is defined to be the union of $P_\alpha$ along with all of its neighbors. This is because the neighborhood $B_{\rho}(P_\alpha, 1/2)$ of $P_\alpha$ is contained in $P_\alpha^\text{nbd}$. Lastly, $\{a_n\}$ is a Cauchy sequence that lies in a finite union of compact sets, hence the sequence has a limit point in $\widetilde E$.

    Applying Lemma \ref{thm:Lemma_flat_R^2} on the complete conformal metric $\rho$ gives an isometry $H \colon \widetilde E \rightarrow \mathbb R^2 = \mathbb C$. By considering the conjugate $\overline{H}$ if necessary, we assume $H$ is orientation-preserving. We now show that $H$ is a conformal isomorphism. To see this, recall that conformal metric $\rho$ is obtained by pulling back the metric $|dz|^2$ on $\mathbb C$ using the charts $h_{\alpha\beta} \colon \operatorname{Int}(P_\alpha \cup e \cup P_\beta) \rightarrow \mathbb C$. As a result of this, $H\circ h_{\alpha\beta}^{-1}$ is an orientation-preserving isometry from an open subset of $\mathbb C$ to another open subset of $\mathbb C$. Note that orientation-preserving isometries between open subsets of $\mathbb C$ take the form $z\mapsto az+b$. This implies that the maps $H\circ h_{\alpha\beta}^{-1}$ are holomorphic for all $\alpha,\beta$. In conclusion, $H$ is a bijective holomorphic map or a conformal isomorphism. 
\end{proof}

\section{Proof of Theorem 1(c) ($\kappa(v)<0$ case)}\label{sec:negative_case}
Let $E$ be a regular polygonal surface whose combinatorial curvature is strictly negative at each vertex. Suppose that each polygon in $E$ has at most $N$ sides. Similar to the proof of \hyperref[thm:Theorem_regular_polygonal_type]{1(b)}, it suffice to consider the case when $E$ is simply-connected because when $E$ is not simply-connected, we can use Lemma \ref{thm:Lemma_lift_polygonal_surface} to get a polygonal surface which is a simply-connected universal covering of $E$. 

We provide a proof of Theorem \hyperref[thm:Theorem_regular_polygonal_type]{1(c)} closely following Oh \cite{oh2005aleksandrov}. We remark that Oh, in fact, proves a more general version for Aleksandrov surfaces satisfying the condition that they can be partitioned into clusters where each cluster contains a definite amount of negative curvature.

The proof can be summarized as follows. First, we show that $E$ is homeomorphic to the plane (Subsection \ref{subsec:E_is_plane}). Then, the crucial observation is that all Jordan domains in $E$ satisfy an isoperimetric inequality (Subsection \ref{subsec:isoperimetric}). Lastly, this observation implies that $E$ is conformally isomorphic to the unit disc $\mathbb D$ by Ahlfors' hyperbolicity criterion (Subsection \ref{subsec:Ahlfors}).
\subsection{Eliminating the elliptic case}\label{subsec:E_is_plane}
In this subsection, we show that $E$ is not the Riemann sphere. Hence, $E$ is conformally isomorphic to either $\mathbb C$ or $\mathbb D$ and $E$ is homeomorphic to the plane. 

\begin{Lemma}[Discrete Gauss-Bonnet]\label{thm:Lemma_Gauss_Bonnet}
    Suppose $G=(V,E,F)$ is a finite graph that is embedded in a compact surface $S$. Let $\chi(S)$ be the Euler characteristic of $S$ and $\kappa(v)$ be the combinatorial curvature at the vertex $v$ in $G$. Then, we have
    \begin{align}
    \begin{split}
        \chi(S) &= \sum_{v\in V} \left( 1 - \frac{\operatorname{deg}(v)}{2} + \sum_{i=1}^{d_v} \frac{1}{k^{(v)}_i}\right) \\
        & = \sum_{v \in V} \kappa(v), \\
    \end{split}
    \end{align}
    where $[k^{(v)}_1, k^{(v)}_2, \dots, k^{(v)}_{d_v}]$ is the vertex-type of the vertex $v$.
\end{Lemma}

It follows immediately from the above lemma that $E$ is not the Riemann sphere. To see this, let $G=(V,E,F)$ be the graph obtained by considering the vertices, edges, and polygons in $E$. Suppose $E$ is the compact Riemann sphere. Then, $E$ has only finitely many polygons and $G$ is a finite graph. Now, Lemma \ref{thm:Lemma_Gauss_Bonnet} implies that 
\[\chi(E)=2=\sum_{v \in V} \kappa(v).\]
This is a contradiction as $\kappa(v) < 0 $ for all $v$ in $E$.

For generalizations of Lemma \ref{thm:Lemma_Gauss_Bonnet} to infinite graphs, see \cite[Theorem 1.3]{devos2007analogue}, \cite[Theorem 2.1]{chen2009gauss}, and \cite[Theorems 2.1, 2.2]{oh2022some}.

\begin{proof}[Proof of Lemma \ref{thm:Lemma_Gauss_Bonnet}]
    Observe that $|V|,|E|$, and $|F|$ can be counted as follows:
    \begin{subequations}\label{eqn:vertex_contribution}
        \begin{align}
            |V| & = \sum_{v \in V} 1 \label{subeqn:V}\\
            |E| & = \sum_{v \in V} \frac{\operatorname{deg}(v)}{2} \label{subeqn:E}\\
            |F| & = \sum_{v \in V} \sum_{i = 1}^{d_v} \frac{1}{k^{(v)}_i}. \label{subeqn:F}
        \end{align}
    \end{subequations}
    To understand equation \ref{subeqn:E}, note that $\operatorname{deg}(v)/2$ is the contribution to the number of edges $|E|$ coming from the vertex $v$. That is, a vertex $v$ has $\operatorname{deg}(v)$ number of edges around it, and each edge is incident with two vertices. Thus, adding $\operatorname{deg}(v)/2$ over all vertices gives the number of edges $|E|$. Equation \ref{subeqn:F} is justified similarly: the sum $\sum 1/k^{(v)}_i$ over the vertex-type $[k^{(v)}_1, \dots, k^{(v)}_{d_v}]$ of the vertex $v$, is the contribution to the number of faces $|F|$ coming from the vertex $v$. Summing over all vertices gives the count of the number of faces $|F|$. 

    Lastly, we combine Euler's formula $\chi(S) = |V| - |E| + |F|$, equation \ref{eqn:vertex_contribution}, and the definition of combinatorial curvature (see equation \ref{eqn:combi_curv}) to deduce Lemma \ref{thm:Lemma_Gauss_Bonnet}.
\end{proof}

\subsection{Isoperimetric inequalities}\label{subsec:isoperimetric}

\begin{Proposition}[Isoperimetric inequality for negatively curved planar graphs]\label{thm:Proposition_Higuchi}
    Let $G=(V,E,F)$ be a planar graph where each face has at least $3$ sides and the degree of each vertex is at least $3$. Suppose the combinatorial curvature $\kappa(v)$ is strictly negative at each vertex. Then, there is a constant $c>0$ such that for every finite subgraph $H=(V(H),E(H),F(H))$ of $G$ we have
    \begin{equation}
        |F(H)| \leq c |E(\partial H)|,
    \end{equation}
    where $E(\partial H) := \{ \{v,w\} \in E(H) \cap E(f') \colon f' \in F(G) \setminus F(H)\}$. \qed
\end{Proposition}

Proposition \ref{thm:Proposition_Higuchi} is proved by Higuchi \cite[Theorem B(i)]{higuchi2001combinatorial} by performing a careful manipulation of Euler's formula.\footnote{Theorem B(i) stated in \cite{higuchi2001combinatorial} also assumes that every face in $G$ is bounded by a cycle and any two faces in $G$ have at most one common edge. But these two assumptions are not necessary and are not used in the proof (see the discussion after Assumption 1.1 in \cite[pp. 221]{higuchi2001combinatorial}).} Various variants of isoperimetric inequalities for negatively curved graphs have been studied. The interested reader may refer to \cite[Theorem 2]{woess1998note} and \cite[Theorem 4]{oh2016strong}.

We now proceed to the proof of the isoperimetric inequality for Jordan regions. To measure the perimeter and area of Jordan regions we shall use the conformal metric $\rho$ introduced in Lemma \ref{thm:Lemma_conformal_metric}. We caution again that this metric is not a smooth Riemannian metric. In local holomorphic charts, it takes the form $\varrho(z)^2|dz|^2$, where $\varrho(z)$ is a continuous function that is smooth and positive away from a discrete set of points. Also, recall that the restriction of this conformal metric $\rho$ to $\operatorname{Int}(P_\alpha)\subseteq E$ gives the standard Euclidean metric $|dz|^2$. Thus, the measurement of $\rho$-length of curves and $\rho$-area of domains inside a polygon $\operatorname{Int}(P_\alpha)\subseteq E$ coincide with the Euclidean length and area, respectively.

Next, we partition $E$ into a collection of triangles $\mathcal{F}$. This collection $\mathcal{F}$ is obtained by taking the collection of polygons $\{P_\alpha\}_{\alpha}$ in $E$, and subdividing each $P_\alpha$ into isosceles triangles by adding one new vertex at the face center (see Figure \ref{fig:star_subdivision}).
\begin{figure}[t!]
    \centering
    {
        \hfill \hfill
        \includegraphics[width=0.25\textwidth]{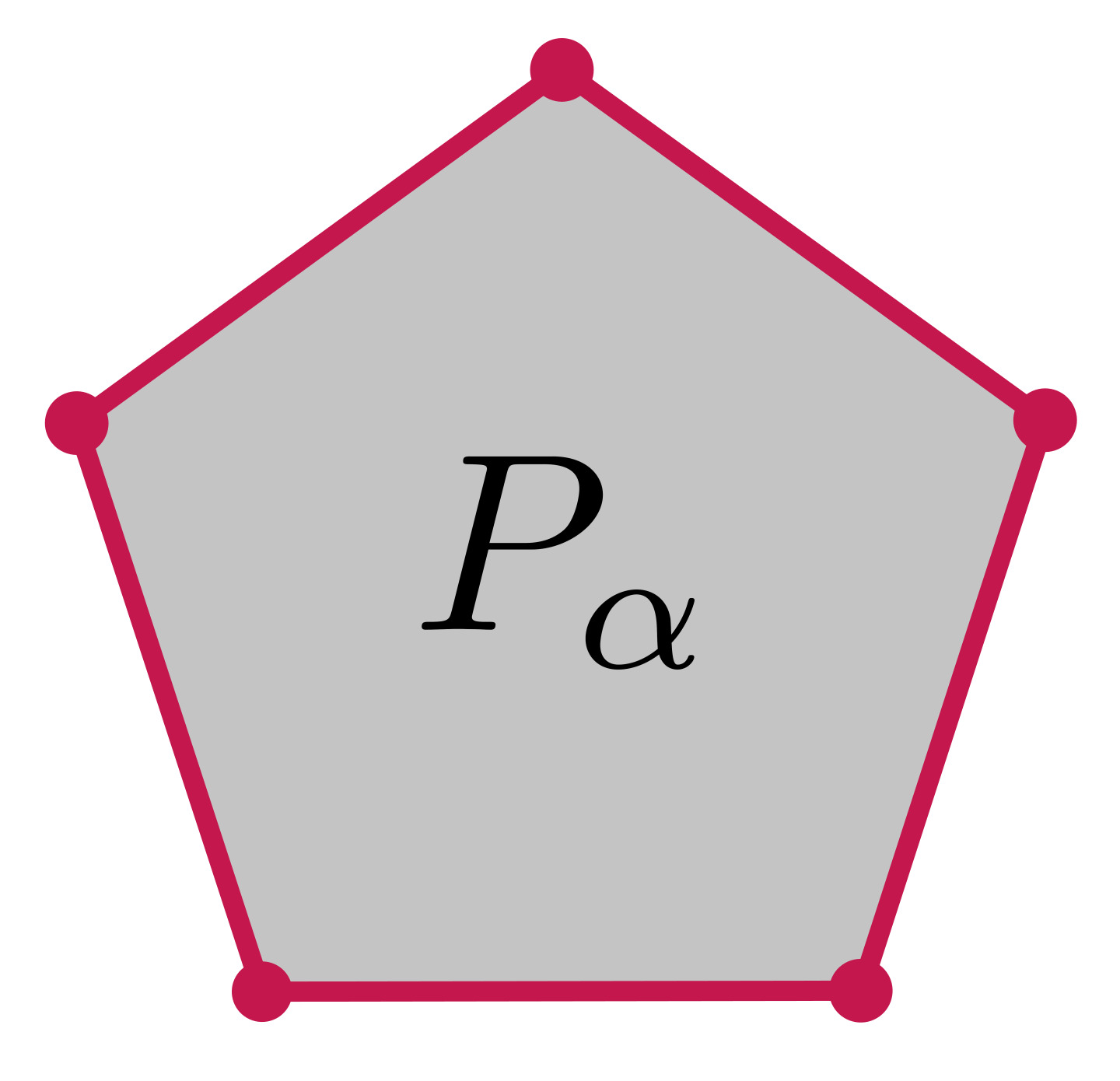} \hfill
        \includegraphics[width=0.25\textwidth]{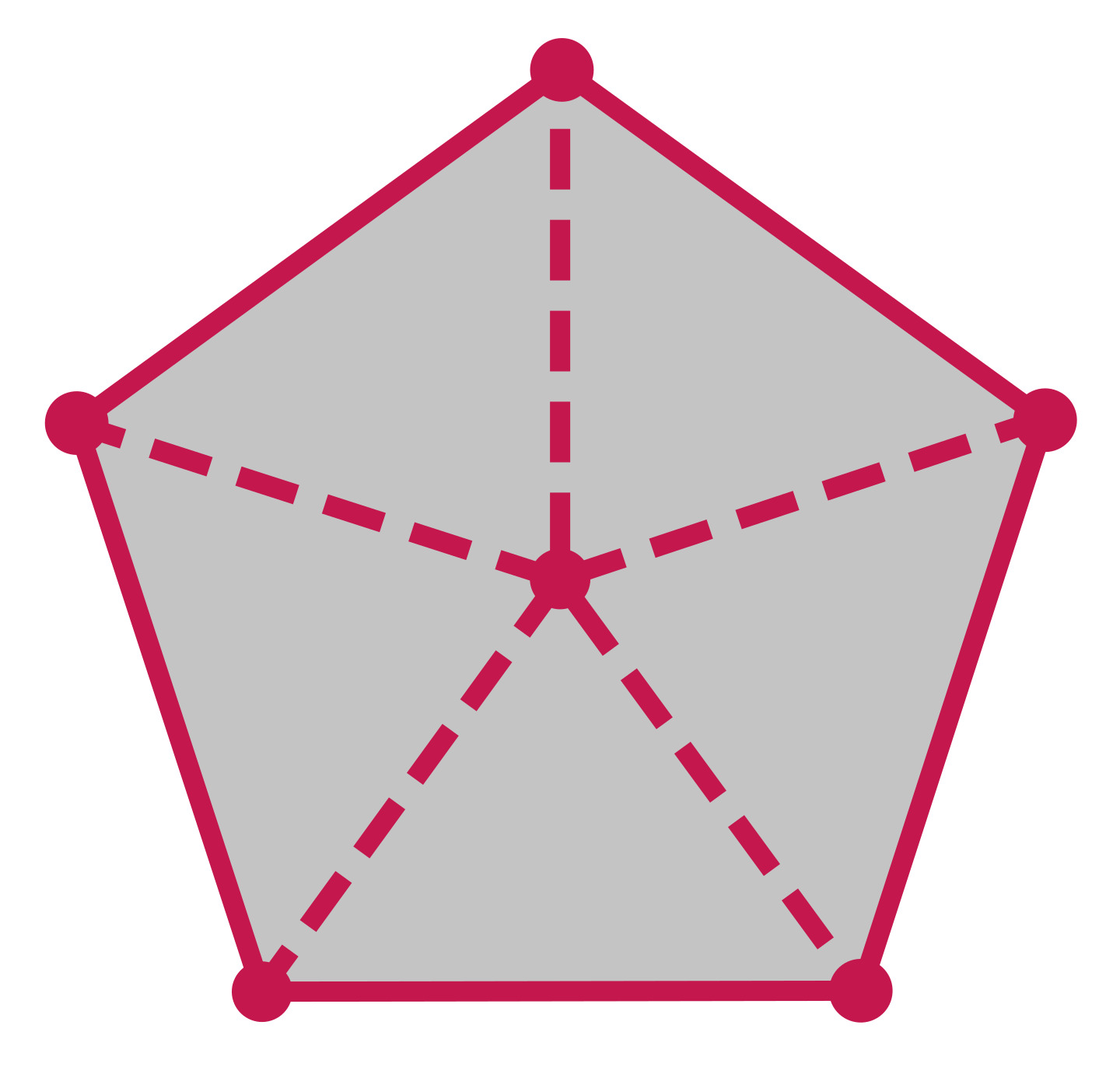}
        \hfill \hfill
    }
    \caption{Star subdivision of a regular $n$-gon into $n$ isosceles triangles.}
    \label{fig:star_subdivision}
\end{figure}

\begin{Proposition}[Intermediate isoperimetric inequality]\label{thm:Proposition_intermediate_isoperimetric}
    Let $E$ be a simply connected regular polygonal surface that has combinatorial curvature strictly negative at each vertex. Suppose $N$ is a constant such that each polygon in $E$ has at most $N$ sides. Let $D$ be an open set consisting of a finite number of triangles in $\mathcal F$. Then there exists a constant $c$ such that 
    \[ \operatorname{Area}_{\rho}(D) \leq c \operatorname{len}_\rho(\partial D).\]
\end{Proposition}

We describe a proof of Proposition \ref{thm:Proposition_intermediate_isoperimetric} closely following \cite[Lemma 3.14]{oh2005aleksandrov}. We remark that this is the only step where we use the assumption about the upper bound $N$ on the number of sides in the polygons of $E$.

\begin{proof}[Proof of Proposition \ref{thm:Proposition_intermediate_isoperimetric}]
    Start by considering the graph $G$ obtained by taking all the vertices, edges, and faces of $E$. This graph is canonically embedded in $E$, which is homeomorphic to the plane. Further, note that each face has at least $3$ sides and each vertex has degree at least $3$ (if the degree were less than $3$, then combinatorial curvature would be positive). Hence, $G$ satisfies the hypothesis of Proposition \ref{thm:Proposition_Higuchi}.

    Next, let $D$ be an open set consisting of finitely many triangles in $\mathcal{F}$. Consider the special case where $D$ is a union of polygons, that is, $D$ is of the form:
    \begin{equation}\label{eqn:D_union_of_polygons}
        D = \operatorname{Int}\left( \bigcup_{i=1}^m P_{i}\right),
    \end{equation}
    for some polygons $P_{i}$. In this case, we apply Proposition \ref{thm:Proposition_Higuchi} on the subgraph $H$ of $G$ obtained by considering all the vertices, edges, and faces of $\{P_{i}\}_{1\leq i \leq m}$. This gives the ``discrete" isoperimetric inequality $|F(H)| \leq c |E(\partial H)|$. Now, to obtain the isoperimetric inequality in terms of the $\rho$-metric, note the following: 
    \begin{align*}
        \operatorname{Area}_\rho(D) =\sum_{i=1}^{m=|F(H)|} \operatorname{Area}_\rho(P_{i}) &\leq \sum_{i=1}^{m=|F(H)|} \frac{N^2}{4\pi}\\
        &= \frac{N^2}{4\pi}|F(H)|\\
        &\leq \frac{N^2c}{4\pi}|E(\partial H)|\\
        &= \frac{N^2c}{4\pi} \sum_{e \in E(\partial H)} \operatorname{len}_\rho(e)= \frac{N^2c}{4\pi} \operatorname{len}_\rho(\partial D).
    \end{align*}
    The inequality $\operatorname{Area}_\rho(P_{i}) \leq N^2/(4\pi)$ and equality $|E(\partial H)| = \sum_{e \in E(\partial H)} \operatorname{len}_\rho(e)$ are both a consequence of the fact that the $\rho$-metric agrees with the Euclidean metric on each $\operatorname{Int}(P_i)$: the area of an $n$-sided unit regular Euclidean polygon is bounded above by $n^2/(4\pi)$ and the Euclidean side length of each $P_i$ is exactly one unit, by construction.    
    
    In the general case where $D$ is not a union of polygons, let $\{Q_i\}_{1\leq i \leq d}$ be the set of all polygons in $E$ that intersect $\partial D$. Consider $D' = \operatorname{Int}(D \cup \bigcup_{i=1}^d Q_i)$ and note that $D \cup \bigcup_{i=1}^d Q_i$ is a union of some finitely many polygons $\{P_i\}_{1\leq i \leq m}$. This means $D'$ is of the special form $D' = \operatorname{Int}\left( \bigcup_{i=1}^m P_{i}\right)$ for which the isoperimetric inequality $ \operatorname{Area}_{\rho}(D') \leq c \operatorname{len}_\rho(\partial D')$ holds true. This gives rise to the isoperimetric inequality for $D$ as follows:  
    \begin{align*}
        \operatorname{Area}_\rho(D) &\leq  \operatorname{Area}_\rho(D')\\
        &\leq c'\operatorname{len}_\rho(\partial D') = c'\cdot |E(\partial D')| \\
        &\leq c'N\cdot \#\{\text{polygons } Q_i \text{ that intersect }\partial D\} \\
        &\leq 2c'N \cdot |E(\partial D)|\\
        &\leq 2c'N \sum_{e\in E(\partial D)} \sqrt{3}\operatorname{len}_\rho(e)= 2\sqrt{3}c'N \operatorname{len}_\rho(\partial D).
    \end{align*}
    The first inequality holds because $D \subseteq D'$. Next, $|E(\partial D')|$ is the number of edges in $\partial D'$ which is at most $N$ times the number of polygons $Q_i$ that intersect $\partial D$. Then, the number of polygons $Q_i$ that intersect $\partial D$ is at most twice the number of edges in $\partial D$. The last inequality $|E(\partial D)| \leq \sum_{e\in E(\partial D)} \sqrt{3}\operatorname{len}_\rho(e)$ is obtained by noting that each edge $e \in E(\partial D)$ has a length of at least $1/\sqrt{3}$. This is because an edge in $E(\partial D)$ is an edge of some isosceles triangle in $\mathcal{F}$ that was obtained by subdividing a unit regular polygon by adding a vertex at the center.
\end{proof}

\begin{Proposition}[Isoperimetric inequality for Jordan domains]\label{thm:Proposition_Jordan_isoperimetric}
    Let $E$ be a simply connected regular polygonal surface that has combinatorial curvature strictly negative at each vertex. Suppose $N$ is a constant such that each polygon in $E$ has at most $N$ sides. Let $D$ be a Jordan domain in $E$. Then, there exists a constant $c$ such that 
    \pushQED{\qed}
    \[ \operatorname{Area}_{\rho}(D) \leq c \operatorname{len}_\rho(\partial D). \qedhere \]
    \popQED
\end{Proposition}

The crucial tool used in the proof of Proposition \ref{thm:Proposition_Jordan_isoperimetric} above is the following lemma.

\begin{Lemma}[Replacing Jordan curves with piece-wise straight lines]\label{thm:Lemma_Jordan_to_st_line}
    If $D$ is a Jordan region in $E$ and $\Delta$ a triangle in $\mathcal{F}$ such that $\operatorname{Int}(\Delta)\cap \partial D \neq \emptyset$ and $D \nsubseteq \Delta$, then we may obtain a region $D'$ by properly adding or subtracting $\Delta$ from $D$ so that the following properties hold for some constant $c$:
    \pushQED{\qed}
    \begin{align*}
        \operatorname{len}_\rho(\partial D') &\leq \operatorname{len}_\rho(\partial D) + c\operatorname{len}_\rho(\partial D \cap \operatorname{Int}(\Delta)) \\
        \operatorname{Area}_\rho(D') &\leq \operatorname{Area}_\rho(D) + c\operatorname{len}_\rho(\partial D \cap \operatorname{Int}(\Delta)).  \qedhere
    \end{align*}
    \popQED
\end{Lemma}

Lemma \ref{thm:Lemma_Jordan_to_st_line} and Proposition \ref{thm:Proposition_Jordan_isoperimetric} are proved in \cite[Lemma 4.13 and pp. 4567]{oh2005aleksandrov}.

\subsection{Ahlfors hyperbolicity criteria}\label{subsec:Ahlfors}

\begin{Proposition}[Ahlfors hyperbolicity criteria]\label{thm:Proposition_Ahlfors}
    Let $X$ be an open simply-connected Riemann surface with a conformal metric $\rho$. Here, by a conformal metric, we mean that $\rho$ takes the form $\varrho(z)|dz|^2$ in any holomorphic chart, where $\varrho(z)$ is a continuous function that is smooth and positive away from a discrete set. If $\rho$ allows a linear isoperimetric inequality for every Jordan region in $X$, then $X$ is conformally equivalent to the unit disc $\mathbb D$.
\end{Proposition}

Note that Theorem \hyperref[thm:Theorem_regular_polygonal_type]{1(c)} now follows as a consequence of Propositions \ref{thm:Proposition_Jordan_isoperimetric} and \ref{thm:Proposition_Ahlfors}. We shall prove Proposition \ref{thm:Proposition_Ahlfors} closely following Hayman \cite[pp. 143-144]{hayman1964meromorphic}.

\begin{proof}[Proof of Proposition \ref{thm:Proposition_Ahlfors}]
    Suppose for contradiction that $X$ is conformally equivalent to $\mathbb C$ via a conformal isomorphism $F: \mathbb C \rightarrow X$. The pull-back metric $F^*(\rho)$ can be written in the form $F^*(\rho) = \alpha(z) |dz|$ for some continuous function $\alpha$ that is positive away from a discrete set. Let $L(r)$ and $A(r)$ denote the length of $\{|z|=r\}$ and the area of $\{|z|\leq r\}$, respectively. First, note that $A(r)$ is positive for all $r>0$ because $\alpha(z)$ is positive almost everywhere with respect to the Lebesgue measure $|dz|^2$. Second, by the linear isoperimetric inequalities for Jordan regions on $X$, we have:
    \begin{equation}\label{eqn:Ahlfors_1}
        A(r) \leq c L(r).
    \end{equation}
    Third, by Cauchy-Schwarz inequality, we obtain:
    \begin{equation}\label{eqn:Ahlfors_2}
        L(r)^2 \leq 2\pi r A'(r).
    \end{equation}
    To see this, note the following calculations:
    \begin{align*}
        (i) \quad L(r)^2 &= \left( \int_0^{2\pi} \alpha(re^{i \theta}) r \, d\theta\right)^2 \\
        &\leq \int_0^{2\pi} (\sqrt{r})^2 \, d\theta \cdot  \int_0^{2\pi} (\alpha(re^{i \theta}) \sqrt{r})^2 \, d\theta \\
        &\leq 2\pi r \int_0^{2\pi} \alpha(re^{i \theta})^2 r \, d\theta \\
        (ii) \quad A'(r) &= \left. \frac{d}{ds}\right|_{s=r} \left( \int_{|z|\leq s} \alpha(z)^2 |dz|^2 \right)\\
        &= \left. \frac{d}{ds}\right|_{s=r} \left( \int_{t=0}^{s} \left( \int_{\theta = 0}^{2\pi} \alpha(te^{i \theta})^2 t\, d\theta\right) dt \right)\\
        &=\int_{\theta = 0}^{2\pi} \alpha(re^{i \theta})^2 r \, d\theta.
    \end{align*}
    Now, combining equations \ref{eqn:Ahlfors_1} and \ref{eqn:Ahlfors_2}, we get $A(r)^2 \leq 2\pi c^2 r  A'(r)$. Next, pick an arbitrary $r_0 \in (0,\infty)$ and integrate $A'(t)/A(t)^2$ over the interval $(r_0,r)$ to obtain:
    \begin{align*}
        \frac{1}{A(r_0)} - \frac{1}{A(r)}  = \int_{r_0}^r \frac{A'(t)}{A(t)^2}dt \geq \frac{1}{2\pi c^2}\int_{r_0}^r \frac{dt}{t} = \frac{1}{2\pi c^2} \log \frac{r}{r_0}
    \end{align*}
    As $A(r)>0$, taking the limit $r \rightarrow \infty$ we obtain $A(r_0) = 0$. This is a contradiction since $A(r)$ is a positive function.
\end{proof}

\appendix

\section{Metric on $r$-spherical surfaces}\label{Appendix_metric}

\begin{proof}[Proof of part (a) of Lemma \ref{thm:Lemma_spherical_metric}]
    Recall the pseudo-metric $d:E_r \times E_r \rightarrow \mathbb R$ is defined to be:
    \[ d(x,y) := \inf \{\len(\gamma) \colon \gamma \text{ is a piece-wise geodesic from } x \text{ to } y \}.\]
    To show $d$ is a metric, we require $d(x,y)>0$ for all pairs of distinct points $x,y \in E_r$. Suppose $x$ lies in the intersection of $r$ polygons $P_1,P_2,\dots,P_r$. For $\epsilon>0$, let $B(x,\epsilon,P_i)$ denote the open ball of radius $\epsilon$ in $P_i$, where distance is measured in the Riemannian metric of $P_i \subseteq S^2_r$. Choose $\epsilon>0$ small enough such that $y$ is not contained in the open neighborhood $B(x,\epsilon) = \bigcup_{i=1}^r B(x,\epsilon,P_i)$. 

    Now, we claim that $d(x,y)\geq \epsilon >0$, or equivalently, any piece-wise geodesic $\gamma$ from $x$ to $y$ has a length of at least $\epsilon>0$. To see this let $\gamma$ be a piece-wise geodesic from $x$ to $y$. As $y \notin B(x,\epsilon)$, there is a $t_0>0$ such that $\gamma([0,t_0)) \subseteq B(x,\epsilon)$ and $\gamma(t_0) \notin B(x,\epsilon)$. Further, there is a partition $0=a_1<a_2<\cdots < a_{m+1}= t_0$ such that each segment $\gamma|_{[a_j,a_{j+1}]}$ lies in a single polygon $P_{k_j}\subseteq \bigcup_{i=1}^r P_i$. Now, using triangle inequality, note the following: 
    \begin{align*}
        \sum_{j=1}^m \len(\gamma_{[a_j,a_{j+1}]}) &\geq \sum_{j=1}^m d_{P_{k_j}}(\gamma(a_j),\gamma(a_{j+1})) \\
        &\geq \sum_{j=1}^m \left( d_{P_{k_j}}(x,\gamma(a_{j+1})) - d_{P_{k_j}}(x,\gamma(a_j)) \right)\\
        &= \sum_{j=1}^m \left( d_{P_{k_j}}(x,\gamma(a_{j+1})) - d_{P_{k_{j-1}}}(x,\gamma(a_j)) \right)\\
        &= d_{P_{k_m}}(x,\gamma(a_{m+1})) \\
        &=\epsilon.
    \end{align*}
    Here, $d_{P_{i}}$ denotes the distance function in the Riemannian metric of $P_i \subseteq S^2_r$. The above calculation shows that $\len(\gamma|_{[0,t_0]}) \geq \epsilon > 0$ for every piece-wise geodesic $\gamma$ joining $x$ and $y$. Hence, the pseudo-metric $d$ is a metric on $E_r$.
\end{proof}

\begin{proof}[Proof of parts (b) and (c) of Lemma \ref{thm:Lemma_spherical_metric}]

    We proceed by using the Hopf-Rinow theorem for path metric spaces. A metric space $(X,d)$ is said to be a \textit{path metric space} if the distance between each pair of points equals the infimum of the $d$-lengths of curves joining the points. The $d$-length of a continuous curve $\sigma : [0,l] \rightarrow X$ in a metric space $(X,d)$ is defined as follows.
    \[ \operatorname{len}_d(\sigma) := \sup \left\{ \sum_{j=1}^m d(\sigma(a_j),\sigma(a_{j+1})) \colon 0=a_1 < a_2 < \cdots < a_{m+1}=l \right\}\]

    \medskip
    \noindent \textbf{Hopf-Rinow theorem (path metric spaces).}\label{thm:Hopf-Rinow}
    \textit{If $(X,d)$ is a sequentially complete, locally compact path  metric space, then
    \begin{enumerate}
        \item each pair of points can be joined by a minimizing geodesic;
        \item every closed and bounded subset is compact.
    \end{enumerate}}

    \medskip
    For a proof of this theorem, see, for example, Section 1.B of \cite[pg. 9]{gromov1999metric} or Proposition I.3.7 of \cite[pg. 35]{bridson2013metric}. Now we put $X=E_r$ with the metric $d$ defined in equation \ref{eqn:spherical_metric} and proceed to show that $(E_r,d)$ is a sequentially complete, locally compact path metric space. Hence, parts (b) and (c) of Lemma \ref{thm:Lemma_spherical_metric} will follow from the above Hopf-Rinow theorem.

    To prove that $(E_r,d)$ is a path metric space we need the following equality:
    \begin{align}\label{eqn:len_d=l}
    \begin{split}
        &\inf \{ \len_d(\sigma) \colon \sigma \text{ is a continuous curve from } x \text{ to } y \} 
        \\&= d(x,y):= \inf \{\len(\gamma) \colon \gamma \text{ is a piece-wise geodesic from } x \text{ to } y \}
    \end{split}
    \end{align}
    The equality follows from two inequalities. First,  if $\sigma$ is a continuous curve from $x$ to $y$, then by definition of $\len_d(\sigma)$ and the use of triangle inequality, we get $\len_d(\sigma) \geq d(x,y)$. Second, if $\gamma:[0,l]\rightarrow E_r$ is a piece-wise geodesic from $x$ to $y$, then $\len_d(\gamma) \leq \len(\gamma)$. To see this, let $0=a_1 < a_2 < \cdots < a_{m+1}=l$ be a partition such that $\len_d(\gamma) \approx \sum_{j=1}^m d(\gamma(a_j),\gamma(a_{j+1}))$. Next, after refining the partition $\{a_j\}$ if necessary, note that each segment $\gamma|_{[a_j,a_{j+1}]}$ lies completely in a polygon $P_j$. This gives $d(a_j,a_{j+1})\leq \len(\gamma|_{[a_j,a_{j+1}]})$ and $\len_d(\gamma) \leq \len(\gamma)$.

    Next, we show that $E_r$ is locally compact. Given a point $x \in E_r$, suppose $x$ lies in the polygons $P_1,P_2,\dots,P_r$. Then, using the notation above, note that for a sufficiently small $\epsilon>0$, the open ball $B(x,\epsilon) = \bigcup_{i=1}^r B(x,\epsilon,P_i)$ is contained in the compact set $\bigcup_{i=1}^r P_i$. 

    Lastly, the proof of the sequential completeness of $(E_r, d)$ follows in the same fashion as the proof of the sequential completeness of $(E, d)$ in Proposition \ref{thm:Proposition_angle_2pi_regular_polygonal_surface}. In conclusion, $(E_r,d)$ is a sequentially complete, locally compact path metric space as required.
\end{proof}

\begin{proof}[Proof of part (d) of Lemma \ref{thm:Lemma_spherical_metric}]
    The interior of each $r$-spherical polygon $\operatorname{Int}(P_\alpha)$ in $E_r$ is, by definition, an open subset of $S^2_r$, that is, there is an embedding $h_\alpha: \operatorname{Int}(P_\alpha) \rightarrow S^2_r$. Hence, $\operatorname{Int}(P_\alpha)$ has a smooth Riemannian metric of constant curvature $+1/r^2$. Proceeding as in the proof of Lemma \ref{thm:Lemma_flat_conformal_metric} and \ref{thm:Lemma_conformal_metric}, we obtain a smooth Riemannian metric $g$ of constant curvature $+1/r^2$ on $E_r\setminus V$, where $V$ is the set of vertices in $E_r$.

    Now, let $d_g$ be the distance function defined by $g$, that is,
    \[ d_g(x,y) = \inf\left\{ \int |\dot{\tau}|_g \, dt \colon \tau \text{ is a piece-wise smooth path from } x \text{ to } y\right\},
    \]
    where $|\cdot|_g$ is norm induced by $g$ on the tangent spaces of $E_r \setminus V$. To show $d_g = d|_{E_r \setminus V}$, first, note that $d_g(x,y) \leq d(x,y)$. This is because, for each piece-wise geodesic $\gamma$ of $E_r \setminus V$, there is a partition $0=a_1<\cdots<a_{m+1}=l$ such that $\gamma|_{[a_j,a_{j+1}]}$ is a geodesic in $P_j$ and $\int_{a_j}^{a_{j+1}} |\dot{\gamma}|_g\, dt=\len(\gamma|_{[a_j,a_{j+1}]})$. Next, to show $d_g(x,y) \geq d(x,y)$, let $\tau$ be a piece-wise smooth curve joining $x$ to $y$. Recall, each point in $E_r\setminus V$ is locally isometric to $S^2_r$. Hence, for each point $z$ in the image of $\tau$, there is a small neighborhood $N_z$ such that any two points in $N_z$ can be connected by a $d_g$-minimizing geodesic. By compactness of the image of $\tau$, there is a partition $0=a_0 < \cdots < a_{m+1} = l$ such that $\tau(a_j)$ and $\tau(a_{j+1})$ can be connected by a minimizing geodesic $\gamma_j$. Concatenating these $\gamma_j$ gives a piece-wise geodesic $\gamma$ from $x$ to $y$ such that: 
    \begin{align*}
        \int_{0}^{l} |\dot{\tau}|_g\, dt  &= \sum_{j=1}^m \int_{a_j}^{a_{j+1}} |\dot{\tau}|_g\, dt \\
        & \geq \sum_{j=1}^m d_g(\tau(a_j),\tau(a_{j+1})) \\
        & = \sum_{j=1}^m \len(\gamma_j) \\
        & = \len(\gamma).        
    \end{align*}
     This shows $d_g(x,y) \geq d(x,y)$ and proves that $d_g = d|_{E_r \setminus V}$.
\end{proof}

\bibliographystyle{amsalpha}
\bibliography{refs}

\providecommand{\bysame}{\leavevmode\hbox to3em{\hrulefill}\thinspace}
\providecommand{\MR}{\relax\ifhmode\unskip\space\fi MR }
\providecommand{\MRhref}[2]{%
  \href{http://www.ams.org/mathscinet-getitem?mr=#1}{#2}
}
\providecommand{\href}[2]{#2}
\begin{thebibliography}{GKPS99}

\bibitem[Bel79]{belyi1979galois}
Gennadii~Vladimirovich Belyi, \emph{On {Galois} extensions of a maximal cyclotomic field}, Izvestiya Rossiiskoi Akademii Nauk. Seriya Matematicheskaya \textbf{43} (1979), no.~2, 267--276.

\bibitem[BH13]{bridson2013metric}
Martin~R Bridson and Andr{\'e} Haefliger, \emph{Metric spaces of non-positive curvature}, vol. 319, Springer Science \& Business Media, 2013.

\bibitem[BR21]{bishop2021non}
Christopher~J Bishop and Lasse Rempe, \emph{Non-compact {Riemann} surfaces are equilaterally triangulable}, arXiv preprint arXiv:2103.16702 (2021).

\bibitem[BS91]{beardon1991circle}
Alan~F Beardon and Kenneth Stephenson, \emph{Circle packings in different geometries}, Tohoku Mathematical Journal, Second Series \textbf{43} (1991), no.~1, 27--36.

\bibitem[BS97]{bowers1997regular}
Philip~L Bowers and Kenneth Stephenson, \emph{A “regular” pentagonal tiling of the plane}, Conformal Geometry and Dynamics of the American Mathematical Society \textbf{1} (1997), no.~5, 58--86.

\bibitem[BS17]{bowers2017conformal}
\bysame, \emph{Conformal tilings {I}: Foundations, theory, and practice}, Conform. Geom. Dyn. \textbf{21} (2017), no.~1, 1--63.

\bibitem[BS19]{bowers2019conformal}
\bysame, \emph{Conformal tilings {II}: Local isomorphism, hierarchy, and conformal type}, Conform. Geom. Dyn \textbf{23} (2019), no.~4, 52--104.

\bibitem[Che09]{chen2009gauss}
Beifang Chen, \emph{The {Gauss}-{Bonnet} formula of polytopal manifolds and the characterization of embedded graphs with nonnegative curvature}, Proceedings of the American Mathematical Society \textbf{137} (2009), no.~5, 1601--1611.

\bibitem[DCFF92]{do1992riemannian}
Manfredo~Perdigao Do~Carmo and J~Flaherty~Francis, \emph{Riemannian geometry}, vol.~6, Springer, 1992.

\bibitem[DM07]{devos2007analogue}
Matt DeVos and Bojan Mohar, \emph{An analogue of the {Descartes}-{Euler} formula for infinite graphs and {Higuchi’s} conjecture}, Transactions of the American Mathematical Society \textbf{359} (2007), no.~7, 3287--3300.

\bibitem[Doy88]{doyle1988deciding}
Peter~G Doyle, \emph{On deciding whether a surface is parabolic or hyperbolic}, Contemporary Mathematics \textbf{73} (1988), 41--49.

\bibitem[EEK82]{edmonds1982regular}
Allan~L Edmonds, John~H Ewing, and Ravi~S Kulkarni, \emph{Regular tessellations of surfaces and (p, q, 2)-triangle groups}, Annals of Mathematics (1982), 113--132.

\bibitem[Ghi23]{ghidelli2023largest}
Luca Ghidelli, \emph{On the largest planar graphs with everywhere positive combinatorial curvature}, Journal of Combinatorial Theory, Series B \textbf{158} (2023), 226--263.

\bibitem[GKPS99]{gromov1999metric}
Mikhael Gromov, Misha Katz, Pierre Pansu, and Stephen Semmes, \emph{Metric structures for {Riemannian} and non-{Riemannian} spaces}, vol. 152, Springer, 1999.

\bibitem[Gri85]{grigor1985existence}
Alexander~Asaturovich Grigor'yan, \emph{On the existence of positive fundamental solutions of the {Laplace} equation on {Riemannian} manifolds}, Matematicheskii Sbornik \textbf{170} (1985), no.~3, 354--363.

\bibitem[Hay64]{hayman1964meromorphic}
Walter~Kurt Hayman, \emph{Meromorphic functions}, vol.~78, Oxford Clarendon Press, 1964.

\bibitem[Hig01]{higuchi2001combinatorial}
Yusuke Higuchi, \emph{Combinatorial curvature for planar graphs}, Journal of Graph Theory \textbf{38} (2001), no.~4, 220--229.

\bibitem[HS95]{he1995hyperbolic}
Zheng-Xu He and Oded Schramm, \emph{Hyperbolic and parabolic packings}, Discrete \& computational geometry \textbf{14} (1995), no.~2, 123--149.

\bibitem[JW16]{jones2016dessins}
Gareth~A Jones and J{\"u}rgen Wolfart, \emph{Dessins d'enfants on {Riemann} surfaces}, Springer, 2016.

\bibitem[Lee97]{lee2006riemannian}
John~M Lee, \emph{Riemannian manifolds: an introduction to curvature}, vol. 176, Springer Science \& Business Media, 1997.

\bibitem[Leo02]{kenneth2002masters}
Dean~Kenneth Leonardi, \emph{On deciding whether a surface is parabolic or hyperbolic}, Master's thesis, 2002.

\bibitem[Mil77]{milnor1977deciding}
John Milnor, \emph{On deciding whether a surface is parabolic or hyperbolic}, The American Mathematical Monthly \textbf{84} (1977), no.~1, 43--46.

\bibitem[NS11]{nicholson2011new}
Ruanui Nicholson and Jamie Sneddon, \emph{New graphs with thinly spread positive combinatorial curvature}, New Zealand J. Math \textbf{41} (2011), 39--43.

\bibitem[Oh05]{oh2005aleksandrov}
Byung-Geun Oh, \emph{Aleksandrov surfaces and hyperbolicity}, Transactions of the American Mathematical Society \textbf{357} (2005), no.~11, 4555--4577.

\bibitem[Oh17]{oh2017number}
\bysame, \emph{On the number of vertices of positively curved planar graphs}, Discrete Mathematics \textbf{340} (2017), no.~6, 1300--1310.

\bibitem[Oh22]{oh2022some}
\bysame, \emph{Some criteria for circle packing types and combinatorial {Gauss}-{Bonnet} theorem}, Transactions of the American Mathematical Society \textbf{375} (2022), no.~2, 753--797.

\bibitem[Old17]{oldridge2017characterizing}
Paul~Richard Oldridge, \emph{Characterizing the polyhedral graphs with positive combinatorial curvature}, Ph.D. thesis, 2017.

\bibitem[OS16]{oh2016strong}
Byung-Geun Oh and Jeehyeon Seo, \emph{Strong isoperimetric inequalities and combinatorial curvatures on multiply connected planar graphs}, Discrete \& Computational Geometry \textbf{56} (2016), no.~3, 558--591.

\bibitem[Sto76a]{stone1976combinatorial}
David~A Stone, \emph{A combinatorial analogue of a theorem of {Myers}}, Illinois Journal of Mathematics \textbf{20} (1976), no.~1, 12--21.

\bibitem[Sto76b]{stone1976correction}
\bysame, \emph{Correction to my paper “{A} combinatorial analogue of a theorem of {Myers}”}, Illinois Journal of Mathematics \textbf{20} (1976), no.~3, 551--554.

\bibitem[Sto76c]{stone1976geodesics}
\bysame, \emph{Geodesics in piecewise linear manifolds}, Transactions of the American Mathematical Society \textbf{215} (1976), 1--44.

\bibitem[Woe98]{woess1998note}
Wolfgang Woess, \emph{A note on tilings and strong isoperimetric inequality}, Mathematical Proceedings of the Cambridge Philosophical Society, vol. 124, Cambridge University Press, 1998, pp.~385--393.

\end{thebibliography}

\end{document}